\documentclass[12pt,a4paper]{amsart}%
\usepackage{imakeidx}
\usepackage{amsmath}
\usepackage{fullpage}
\usepackage{comment}
\usepackage{MyCommA}
\usepackage[final]{pdfpages}
\usepackage{tikz-cd}
\makeindex
\usetikzlibrary{calc}

\newcommand{\DimaC}[1]{{{#1}}}
\newcommand{\DimaB}[1]{{{#1}}}
\newcommand{\DimaA}[1]{{#1}}
\newcommand{\Dima}[1]{{{#1}}}
\newcommand{\Rami}[1]{{{#1}}}
\newcommand{\RamiA}[1]{{{#1}}}
\newcommand{\RamiB}[1]{{{#1}}}

\newcommand{\NextVer}[1]{}
\newcommand{\sub}{\subset}
\newcommand{\noleft}{\left.\kern-\nulldelimiterspace}

\newcommand{\alp}{\alpha}
\newcommand{\lam}{\lambda}
\newcommand{\valua}{\operatorname{val}}

\newtheorem{innercustomthm}{Theorem}

\newenvironment{customthm}[2][]{%
  \renewcommand{\theinnercustomthm}{#2}
  \begin{innercustomthm}[#1]
}{%
  \end{innercustomthm}%
}

\newcommand{\bfCom}{\mathbf{Com}}

\begin{document}
	
	\author[Aizenbud]{Avraham Aizenbud}
	\address{Avraham Aizenbud,
		Faculty of Mathematical Sciences,
		Weizmann Institute of Science,
		76100
		Rehovot, Israel}
	\email{aizenr@gmail.com}
	\urladdr{https://www.wisdom.weizmann.ac.il/~aizenr/}
	
	\author[Gourevitch]{Dmitry Gourevitch}
	\address{Dmitry Gourevitch,
		Faculty of Mathematical Sciences,
		Weizmann Institute of Science,
		76100
		Rehovot, Israel}
	\email{dimagur@weizmann.ac.il}
	\urladdr{https://www.wisdom.weizmann.ac.il/~dimagur/}

\author[Kazhdan]{David Kazhdan}
	\address{David Kazhdan,
    Einstein Institute of Mathematics, Edmond J. Safra Campus, Givaat Ram The
Hebrew University of Jerusalem, Jerusalem, 91904, Israel}
	\email{david.kazhdan@mail.huji.ac.il}
	\urladdr{https://math.huji.ac.il/~kazhdan/}
	
\author[Sayag]{Eitan Sayag}
	\address{Eitan Sayag,
    Department of Mathematics, Ben Gurion University of the Negev, P.O.B. 653,
Be’er Sheva 84105, ISRAEL}
	\email{eitan.sayag@gmail.com}
	\urladdr{www.math.bgu.ac.il/~sayage}

	\date{\today}

	%
	%
	%
	%
	%
	%
	%
	%
	

\title{Orbital integral bounds the character for cuspidal representations of $\GL_n(\F_{\ell}((t)))$}
\maketitle
\begin{abstract} 
We prove that the character of an irreducible cuspidal representation of $\GL_n(\F_{\ell}((t)))$ is locally bounded up to a logarithmic factor by the orbital integral of a matrix coefficient of this representation.

The characteristic $0$ analog of this result is part of the proof of the 
celebrated Harish-Chandra's integrability theorem.

In a sequel work \cite{AGKS3} we use this result in order to prove a positive characteristic analog of Harish-Chandra's integrability theorem under some additional assumptions.
\end{abstract}

\tableofcontents 
\section{Introduction}

Throughout the paper we fix a non-Archimedian local field \tdef{$F$}
of arbitrary characteristic. Denote by \tdef{$\ell$} the size of the residue field of $F$. All the algebraic varieties and algebraic groups that we will consider are defined over $F$. 
We will also fix a natural number \tdef{$n$} and set $\mdef{\bfG}=(\GL_n)_F$.
 Denote $\mdef{G}=\bfG(F)$.

 We will denote by $\mdef{C^{-\infty}}(G)$ the space of generalized functions on $G$,
 \DimaA{{\it i.e.} functionals on the space of smooth compactly supported measures}. 
\subsection{Orbital integrals}
Our main result \DimaA{involves the notion of }the orbital integral of a function on $G$. Let us first define this notion:

Let $G^{rss}$ be the collection of regular semisimple elements in $G$. 
\begin{itemize}
    \item Denote by $\mdef{\mu_{G}}$ the Haar measure on $G$ normalized such that the measure of a maximal compact subgroup in $G$ is 1. 
    \item  For $x\in G^{rss}$ denote by $\mdef{\mu_{G_x}}$ the Haar measure on the torus $G_x$ normalized such that the measure of the maximal compact subgroup of $G_x$ is 1. 
\item For $x\in G^{rss}$ denote by $\mdef{\mu_{G\cdot x}}$ the $Ad(G)$-invariant measure on the conjugacy class $\RamiA{G\cdot x:=}Ad(G)\cdot x$ 
that corresponds to the measures $\mu_G$ and   $\mu_{G_x}$ under the identification
$Ad(G)\cdot x\cong G/G_x$. 

\item Let $f\in C^{\infty}(G)$ have compact support modulo the center of $G$. Let $\mdef{\Omega(f)}:G^{rss}\to \C$ be the function defined by $\Omega(f)(x)=\int f|_{G\cdot x}\mu_{G\cdot x}$. 
\end{itemize}

\subsection{Main results}
For an irreducible representation $\rho$ of $G$ we denote by $\mdef{\chi_\rho}$ its character, which is a generalized function on $G$.
\RamiA{Our main result consists of a bound on this character in terms of the orbital integral of a function on $G$. In order to formulate the bound we need some notation:}
\begin{itemize}
    \item For $x\in G^{rss}$ denote by $\mdef{\Delta}(x)$ the discriminant of the characteristic polynomial of $x$.
\item For $x\in G$ let $\mdef{ov_{G}}(x):=\max(\max_{i,j}(-\valua(x_{ij})),\valua(\det(x)))$
where $x_{ij}$ are the entries of $x$.
\item For $x\in G^{rss}$ let $\mdef{ov_{G^{rss}}}(x)=\max(ov_G(x),\valua(\Delta(x)))$.
\end{itemize}

\begin{introtheorem}[\S \ref{subsec:Idea}]\label{thm:main}
    Let $\rho$ be a cuspidal irreducible representation of $G$. Let $m$ be a matrix coefficient of $\rho$  $m(1)\neq 0$. Then there exists a polynomial $\mdef{\alp^{\rho\RamiA{,m}}}\in \mathbb{N}[t]$ 
such that for every $\eta\in C^{\infty}_c(G)$ we have 
    $$|\langle \chi_{\rho},\eta\cdot \mu_G\rangle|< \langle f\cdot \Omega(|m|),(|\eta|\cdot \mu_G)|_{G^{rss}}\rangle,$$
    where $f \in C^{\infty}(G^{rss})$ is defined by 
    $f(g)=\alp^{\rho\RamiA{,m}}(ov_{G^{rss}}(g)).$
\end{introtheorem}
\begin{remark*}
    A priori, the right hand side of the above inequality can be infinity.
We interpret the statement in that case as void. 
\end{remark*}
\subsection{Background and motivation}

When the characteristic of $F$ is zero, \Cref{thm:main} is proven in \cite[page 102]{HC_VD}\footnote{This is an immediate corollary of the 3rd displayed formula in \cite[page 102]{HC_VD} together with the first equality in (ii) in that page.}. This is an important step in the proof of Harish-Chandra's integrability theorem: {\it ``The character of an irreducible cuspidal representation of a $p$-adic reductive group is given by a locally integrable function"}, \DimaA{\cite{HC_VD}}. 
The proof in \cite[page 102]{HC_VD}, as well as our proof of \Cref{thm:main}, is based on the fact that averaging of cuspidal functions on $G$ is bounded (up to a logarithmic factor) by their orbital integral. See \Cref{thm:Om.bnd.Av} below.

This fact (in characteristic $0$) is also an important step in the proof of Harish-Chandra's integrability theorem for general (not necessarily cuspidal) irreducible representations.

In a sequel work \cite{AGKS3} we use \DimaA{\Cref{thm:main}} in order to prove an analog of Harish-Chandra's integrability theorem for cuspidal representations of $\GL_n(\F_\ell((t)))$ under some additional assumptions.

\subsection{Idea of the proof}\label{subsec:Idea}
In our argument we will use the following language. 
    
Several statements in this paper will concern the existence of certain polynomials in $\bN[t]$ that satisfy some conditions. In the formulation of each such statement we assign a name for the corresponding polynomial. It is implied that after each such statement we fix such a polynomial and we will refer to it later by this name. Nothing significant will depend on the choices of these polynomials.

We note that in many of the statements one can actually choose this polynomial to be a linear function, but this is not essential to our argument.

Following \cite{HC_VD} our proof can be divided into 2 steps:
\begin{enumerate}
    \item The character of $\rho$ is, up to a scalar, the (weak) limit of the sequence of  functions $\cA_i(m)$, where $m$ is a  matrix coefficient of $\rho$, and $\cA_i(m)$ is the averaging of $m$ w.r.t. a ball \DimaA{$G_i$} in $G$. See \Cref{prop:ChartAv} below.
    
    \item Given $x\in G^{rss}$ one can bound \DimaA{all} $\cA_i(m)(x)$ in terms of $ov_{G^{rss}}(x)$ and $\Omega(m)(x)$\DimaA{, uniformly in $i$}.
\end{enumerate}

We now provide a formal description of these ingredients.
In order to formulate the first ingredient, let us define the notion of averaging:

\begin{definition}\label{def:Av}
Denote
\begin{itemize}
    \item  \tdef{$Z(G)$} to be the center of $G$.
    \item $\mdef{G^{ad}}:=G/Z(G)$.
    \item $\mdef{G_i}:=\{x\in G|ov_G(x)\leq i\}$.
    \item $\mdef{(G^{ad})_i}$ to be the image of $G_i$ under the map $G\to G^{ad}$.
    \item $\mdef{\mu_{Z(G)}}$ to be the Haar measure on $Z(G)$ normalized such that the measure of the maximal compact subgroup of \DimaA{$Z(G)$} is $1$.
\item $\mdef{\mu_{G^{ad}}}$ to be the Haar measure on $G^{ad}$ corresponding to $\mu_G$ and $\mu_{Z(G)}$.
    \item 
For a function $f\in C^{\infty}(G)$, denote its averaging $\mdef{\cA_i}(f)\in C^{\infty}(G)$ by 
$$\mdef{\cA_i(f)(x)}:=\int_{(G^{ad})_i}f(Ad(g)x)dg,$$
where $dg$ is the Haar measure $\mu_{G^{ad}}$.
\end{itemize}
\end{definition}


\DimaA{Let us recall the notion of matrix coefficient of a representation $(\rho,V_{\rho})$. For a pair  $v \in V_{\rho}, \varphi \in \widetilde{V_{\rho}}$, the corresponding matrix coefficient is a smooth function on $G$ defined by $m_{v,\varphi}(g)=\varphi(\rho(g)v)$.}

We can now formulate the formula for the character of a cuspidal representation:
\DimaA{
\begin{thm}[{\S\ref{ssec:pf.ChartAv}, cf. \cite[Theorem 9]{HC_VD}}]\label{prop:ChartAv}
  Let $(\rho,V_{\rho})$ be an irreducible cuspidal representation of $G$. 
  \DimaA{Then there exists a positive number $d(\rho)$ such that for every 
   matrix coefficient $m$ of $\rho$ we have}
  $\cA_i(m)\underset{i\to \infty}{\longrightarrow} \frac{m(1)}{d(\rho)}\chi_{\rho}$, where the convergence is in the weak topology on $C^{-\infty}(G)$.
\end{thm}
In fact,  $d(\rho)$ is the formal dimension of $\rho$ (see \cite[Theorem 1]{HC_VD}).}

\DimaA{The proof of \Cref{prop:ChartAv} in \cite{HC_VD} is valid in arbitrary characteristic. However, the result is not formulated there in this language. For completeness we include the proof of \Cref{prop:ChartAv} in \S\ref{ssec:pf.ChartAv} below.}

In order to  formulate the second ingredient we need the notion of a cuspidal function on $G$:
\begin{definition}
    Let $f\in C^{\infty}(G)$. 
    \begin{itemize}
        \item \RamiA{We say}
    that $f$ is \RamiA{\tdef[cuspidal function]{cuspidal}} if for any unipotent radical $U$ of a proper  parabolic subgroup  of $G$ and
 any $x\in G$ the function $h:U\to \C$ given by $h(u):=f(\RamiA{ux})$ is compactly supported and $$\int h\mu_U=0,$$ where $\mu_U$ is a Haar measure on $U$.    
        \item We denote the collection of cuspidal functions by $\mdef{C^\infty(G)^{cusp}}$.
    \end{itemize}
\end{definition}
Now we can formulate the second ingredient:

\begin{introtheorem}[\S\ref{subsec:PfStab}]
    \label{thm:Om.bnd.Av}
For any $m\in C^{\infty}(G)^{cusp}$ which has compact support modulo the center, there exists a polynomial $\mdef{\alpha^{\RamiA{m}}}\in\bN[t]$ such that 
for any $x\in G^{rss}$ 
we have
$$|\cA_i(m)(x)|\leq \alpha^{\RamiA{m}}(ov_{G^{rss}}(x))\Omega(|m|)(x).$$
\end{introtheorem}

\Cref{thm:main} follows now from \Cref{thm:Om.bnd.Av} and \Cref{prop:ChartAv} using the \RamiA{standard} fact that a matrix coefficient of a cuspidal representation is cuspidal. 

\subsubsection{Idea of the proof of \Cref{thm:Om.bnd.Av}}

The proof of \Cref{thm:Om.bnd.Av} is based on the following 2 ingredients:
\begin{enumerate}
    \item For any given $x\in G^{rss}$, the sequence $\cA_i(m)(x)$ stabilizes. Moreover, there is an effective way to bound the time needed to achieve saturation in terms of $ov_{G^{rss}}(x)$.
    \item One can bound $\cA_i(m)(x)$ in terms of $ov_{G^{rss}}(x),i$ and $\Omega(m)(x)$.
\end{enumerate}
\Rami{Both ingredients involve uniform work on $\RamiA{x\in}G^{rss}$. 
We use the theory of norms developed in \cite{Kot} in 
order to work uniformly on algebraic varieties. Then we prove some bounds on these norms on several algebraic varieties related to $G^{rss}$ -- see \S\ref{sec:norm.bnd}.}

Let us now describe these ingredients \Rami{in more details}.
The first one is the following \RamiA{stabilization result}:
\begin{introtheorem}[\S \ref{subsec:PfStab}]   
\label{thm:stab}
\Rami{For any $m\in C^{\infty}(G)^{cusp}$ which has compact support modulo the center,} there exists a polynomial $\mdef{\alpha^{m}_{ad-stab}}\in\bN[t]$ 
such that for every $x\in G^{rss}$ and every $i>i_0:=\alpha^{m}_{ad-stab}(ov_{G^{rss}}(x))$ we have 
    $$\cA_{i}(m)(x)=\cA_{i_0}(m)(x).$$
\end{introtheorem}
Sections \ref{ssec:norm.bnd}-\ref{subsec:PfStab} are dedicated to the proof of this theorem.
The proof itself is in \S \ref{subsec:PfStab}.
Let us briefly explain the idea \RamiA{of the} proof.
\begin{enumerate}
    \item For any $x$ in $G$ we consider the adjoint action map $\mdef{\phi_x}:G^{ad}\to G$ defined by $\phi_x([g])=gxg^{-1}$. Here $g$ is a representative in $G$ of a class $[g]\in G^{ad}$.
    
    We study the averaging of $m$ w.r.t. the adjoint action using the averaging of $\phi_x^{*}(m)$ w.r.t. the multiplication action.
    \item Note that while $m$ is a cuspidal function (i.e. its integrals over cosets of unipotent radicals of proper parabolic subgroups of $G$ vanish), the function $\phi_{\Rami{x}}^*(m)$ is not cuspidal. However, it turns out that some of the cuspidality survives. Namely, \Rami{we show, in \Cref{lem:cusp.pull} below, that} if $M<G$  is a Levi subgroup and $x\in M\cap G^{rss}$ then $\phi_x^*(m)$ is cuspidal w.r.t. parabolic subgroups corresponding to elements of the center $A$ of  $M$. We call this property $A$-cuspidality, see \Cref{def:A.cusp} below.
    \Rami{Moreover, if $x$ is elliptic in $M$ then $\phi_x^*(m)$ has compact support modulo $A$.}
    \item We prove a version of \Cref{thm:stab} for $A$-cuspidal functions \Rami{with compact support modulo $A$}, where the averaging is taken w.r.t. the multiplication action. See \Cref{thm:cusp.stab} below. In view of the previous step this is already enough in order to give the stabilization for each $x$ separately. However, we need a more uniform result.
    \item For elliptic $x\in M$ which is also an element of $G^{rss}$, we bound the support of  $\phi_x^*(m)$ modulo-$A$ in terms of the support of $m$. See \Cref{lem:pull} below. The proof of this lemma is based on the results of \S\ref{sec:norm.bnd} and it is essentially different from the proof in the zero characteristic case.
    \item The steps above give us \Cref{thm:stab} for the collection of all the elliptic elements in all the standard  Levi subgroups of $G$, which are regular semi-simple in $G$.
    \item In order to complete the proof we use the theory of Norm Descent Property developed in \cite{Kot}. We prove the Norm Descent Property of a certain map  (see \Cref{lem:Lev.NDP} below). This allows us to enhance the results above to obtain uniformity on $G^{rss}.$
\end{enumerate}

The second ingredient \DimaC{in the proof of \Cref{thm:Om.bnd.Av}} is the following:
\begin{introtheorem}[\S \ref{subsec:Pf.bnd.A}]\label{thm:bnd.A}
There exists a polynomial $\mdef{\alpha_{av}}\in \bN[t]$  
such that
for any
\begin{itemize}
    \item $m\in C^{\infty}(G)$ which has compact support modulo the center 
    \item 
$x\in G^{rss}$ 
\item $i\in \bN$ 
\end{itemize}
we have
$$|\cA_i(m)(x)|\leq \alpha_{av}(i+ov_{G^{rss}}(x))\Omega(|m|)(x).$$
\end{introtheorem}    
\Rami{We prove this theorem  using the results of \S\ref{sec:norm.bnd}.}
\subsection{Comparison to 
the characteristic zero case}
Our proof of \Cref{thm:Om.bnd.Av} is similar to 
the original Harish-Chandra's argument in the characteristic $0$ case with the following essential difference:
in the characteristic zero case, one can work for each torus separately, since there \RamiA{are} finitely many tori up to conjugation. This is not the case in positive characteristic. Therefore, we need to give more uniform bounds. For this we use the theory of norms developed in \cite{Kot} and prove uniform bounds on certain norms, see \S \ref{sec:norm.bnd}.

In more details, in the characteristic zero case, one can replace Theorems \ref{thm:stab} and \ref{thm:bnd.A} with the following less uniform versions: 

\begin{thm}[cf. {\cite[page 101]{HC_VD}}\footnote{the first display formula of page \cite[page 101]{HC_VD} gives a slightly weaker statement, but the proof of this formula in fact gives the theorem. }]\label{thm:stab.HC}
\Rami{Let $m\in C^{\infty}(G)^{cusp}$ be a cuspidal function which has compact support modulo the center}. Let $\Gamma<G$ be a maximal (not necessarily split)  torus. Then there exists a polynomial $\alpha^{m,\Gamma}_{ad-stab}\in\bN[t]$
such that for every:
\begin{itemize}
    \item $\gamma\in \Gamma\cap G^{rss}$
    \item $y\in G$
    \item $i>i_0:=\alpha^{m,\Gamma}_{ad-stab}(ov_{\Gamma\cap G^{rss}}(\gamma)+ov_G(y))$
\end{itemize}  we have 
    $$\cA_{i}(m)(y\gamma y^{-1})=\cA_{i_0}(m)(y\gamma y^{-1})$$
\end{thm}

\begin{thm}[cf {\cite[2nd displayed formula in page 102]{HC_VD}}] \label{thm:bnd.A.HC}
Let $\Gamma<G$ be a maximal torus.
Then there exists a polynomial $\mdef{\alpha_{av}^{\Gamma}}\in\bN[t]$ such that
for any
\begin{itemize}
    \item $m\in C^{\infty}(G)$ which has compact support modulo the center 
    \item 
$\gamma\in \Gamma \cap G^{rss}$
    \item 
$y\in G$
\item $i\in \bN$ 
\end{itemize}
we have
$$|\cA_i(m)(y\gamma y^{-1})|\leq \alpha_{av}^{\DimaA{\Gamma}}(i+ov_{\Gamma \cap G^{rss}}(\gamma)+ov(y))\Omega(|m|)(y\gamma y^{-1}) .$$
\end{thm}

Harish-Chandra's proof of these 2 theorems works (with minor changes) for the positive characteristic case. However, while in the characteristic zero case these 2 theorems imply Theorems \ref{thm:stab} and \ref{thm:bnd.A}, this is no longer true in positive characteristic.
\RamiA{The reason is that, in this case, there are infinitely many conjugacy classes of tori.}

Therefore we use a different argument. Our argument \RamiA{is based on \cite{Kot},  \S \ref{sec:norm.bnd}, and ideas of the original Harish-Chandra's argument.}

\subsubsection{The role of the assumption $\bfG=\GL_n$}\label{sssec:gl}
We used the assumption $\bfG=\GL_n$ in order to make all explicit computations easier. However, our argument does not use any statement that inherently depend on this assumption (such as existence of mirabolic subgroup, stability of adjoint orbits, or the Richardson property of all nilpotent orbits).
\Rami{
\subsection{The Lie algebra case}
Harish-Chandra used the characteristic $0$ case of \Cref{thm:main} not only to prove integrability of characters of cuspidal representations, but also to  prove integrability of characters of general irreducible representations, see \cite{HC_SD}. More precisely, he used a Lie algebra version of this theorem. In positive characteristic, like in characteristic $0$, the proof of \Cref{thm:main} also fits its Lie algebra version.
Specifically one can prove the following:
\begin{customthm}[\S\ref{sec:lie}]{\ref{thm:main}'}
\label{thm:main.lie}
    Let $x\in \fg\fl_n(F)$ be an elliptic (regular semi-simple) element. Let $\mu$ be an $ad(G)$-invariant measure on $\fg\fl_n(F)$ supported in  $ad(G)x$. Let $\hat \mu\RamiA{\in C^{-\infty}(\fg\fl_n)}$ be its Fourier transform. 
    
    Then there exists a polynomial $\alp^\RamiA{\mu}\in\mathbb{N}[t]$
such that for every compact  open $B\subset \fg\fl_n(F)$ there exists $m\in C_c^\infty(\fg\fl_n(F))$ such that for every 
$\eta\in C^{\infty}_c(B)$ we have 
    $$|\langle \hat \mu,\eta\cdot \mu_\fg\rangle|< \langle f\cdot \Omega(m),(|\eta|\cdot \mu_\fg)|_{\fg^{rss}}\rangle,$$
    where $f \in C^{\infty}(\fg^{rss})$ is defined by 
    $f(x)=\alp^\RamiA{\mu}(ov_{\fg^{rss}}(x)),$ and $ov_{\fg^{rss}}$ is defined analogously to $ov_{G^{rss}}$.
\end{customthm}

In \S\ref{sec:lie} we explain how to \RamiA{adapt} the proof of \Cref{thm:main} in order to \RamiA{prove} \DimaA{\Cref{thm:main.lie}}. We also formulate there a Lie algebra version of \Cref{thm:Om.bnd.Av}.
}
\subsection{Structure of the paper}

In \S \ref{sec:Con-Nota} we fix some conventions and notation.

In \S \ref{ssec:pf.ChartAv} we prove \Cref{prop:ChartAv}.

In \S \ref{sec:Norms} we give an  overview of the theory of norms on algebraic varieties over local fields developed in \cite[\S18]{Kot}. In particular we recall the notion of norms and of the Norm Descent Property (NDP) of algebraic maps. 

In \S \ref{sec:norm.bnd}, we establish \RamiB{results} showing that a certain function defines a norm, and establish the NDP properties of certain maps (see  Lemmas \ref{lem:phi}, \ref {lem:Lev.NDP} and \Cref{prop:norm}). \RamiA{This  preparation is needed \RamiB{in order to make our bounds} more effective and thus torus independent.}

In \S \ref{subsec:Pf.bnd.A} we prove \Cref{thm:bnd.A}.

In \S \ref{ssec:norm.bnd} \RamiA{we adapt the results of \S \ref{sec:norm.bnd} to fit the needs of \Cref{thm:stab}}.

In \S \ref{subsec:Acusp} we \RamiA{discuss} the notion of $A$-cuspidal function where $A$ is the  center of a standard Levi-subgroup of $G$. This is ``what survives" from cuspidality when we pull a cuspidal function w.r.t. the adjoint action. The goal of this section is to prove \Cref{thm:cusp.stab} which is an analog of \Cref{thm:stab} for $A$-cuspidal functions.

In \S \ref{subsec:PfStab} we prove \RamiB{Theorems \ref{thm:stab} and \Cref{thm:Om.bnd.Av}}.

In \S \ref{sec:lie} we \Rami{\RamiA{discuss the} Lie algebra version of the main results. In particular we prove \Cref{thm:main.lie}, and formulate and prove \Cref{thm:Om.bnd.Av.lie}, which is a Lie algebra version of \Cref{thm:Om.bnd.Av}.}

\subsection{Acknowledgments} 
 During the preparation of this paper, A.A., D.G. and E.S. were partially supported by the ISF grant no. 1781/23. 
D.K. was partially supported by an ERC grant 101142781.

\section{Conventions and Notation}\label{sec:Con-Nota}
\subsection{Conventions}\label{ssec:conv}
\begin{enumerate}
    \item By a \tdef{variety} we mean a reduced scheme of finite type over $F$. 
    \item When we consider a fiber product of varieties, we always consider it in the category of schemes. \Rami{We use set-theoretical notations to define subschemes, whenever no ambiguity is possible.} 
    \item We will usually denote algebraic varieties by bold face letters (such as $\bfX$) and the spaces of their $F$-points by the corresponding usual face letters (such as $X:=\bfX(F)$). We use the same conventions when we want to interpret vector spaces as algebraic varieties.

    \item We will use the same letter to denote a morphism between algebraic varieties and the corresponding map between the sets of their $F$-points.
    \item We will use the symbol \tdef{$\square$} in a middle of a square diagram  in order to indicate that the square is Cartesian. 
    \item By an \tdef{$F$-analytic manifold} we mean an analytic manifold over $F$ in the sense  of \cite{Ser92}.
    \item\label{ssec:conv:5} As we are proving a theorem for $\GL_n$, many of the objects that we consider depend on the parameter $n$. Since we fixed $n$\RamiA{,} our notations will usually not include $n$. However, if we want to consider a certain object for different values of $n$, we will put these values as left supper scripts. For example $\mdef{{}^{k}\bfG}:=\GL_k$ and other uses of left-super-index such as $\mdef{{}^{k}\bfC}, \mdef{{}^{k}\bfS}$. 
    \item\label{ssec:conv:6} If the left superscript is a tuple of natural numbers then we refer to the product of the corresponding objects. For example $$\mdef{{}^{(k_1,k_2)}\bfG}:={}^{k_1}\bfG\times {}^{k_2}\bfG:=\GL_{k_1}\times \GL_{k_2}$$
    \item A \tdef{standard Levi subgroup} of $G$ (resp. $\bfG$)  is a subgroup consisting of block diagonal matrices in $G$ (resp. $\bfG$) with respect to a certain block partition.
    \item A \tdef{standard torus} of $G$ (resp. $\bfG$) is the center of a standard Levi subgroup of $G$ (resp. $\bfG$).    
    \item \Rami{When no ambiguity is possible  we will denote the adjoint action simply by $``\cdot"$.}
    \item \DimaB{We will use the symbol \tdef{$<$} to denote the (not necessarily proper) containment relation for groups}.
\end{enumerate}

\subsection{Notations}

We denote by:
\begin{enumerate}
\item $\mdef \bfT<\bfG$ the maximal standard torus. 
\item \RamiA{$T:=\bfT(F)$.}
\item For a composition $\lambda$ of $n$ denote by \tdef{$\bfT_\lambda$} the standard torus corresponding to this composition. \RamiA{Denote also $T_\lambda:=\bfT_\lambda(F)$.}
\item For a group (or an algebraic group) $H$ we denote by \tdef[$Z(\cdot)$]{$Z(H)$} the center of $H$.
    \item $\mdef{O_F}$ the ring of integers in $F$.
    \item $\mdef{K_0}:=\bfG(O_F)<G$ with respect to the standard $O_F$-structure on $\bfG$.
    \item $\mdef{K_i}<K_0$ the $i$-\Rami{th} congruence  subgroup.
    \item $\mdef{K_i^{ad}}=K_i/Z(K_i)$.
    \item $\mdef{\bfG^{ad}}:=\bfG/Z(\bfG)$. Note that $G^{ad}\RamiA{\lneq} \bfG^{ad}(F).$
        \item \tdef{$\bfC$}  $-$ the variety of monic polynomials of degree $n$ that do not vanish at $0$. We will identify it with $\bG_m\times \A^{n-1}$. 
            We equip $\bf C$ with a group structure using this identification.
        \item $\mdef{C}:=\bfC(F)$.   
        \item $\mdef{p}:\bfG\to \bfC$ $-$ the Chevalley map, i.e. the map that sends a matrix to its characteristic polynomial.
        \item  $\mdef{\mu_C}$ - the Haar measure on $C$, given by the identification $C\cong F^\times \times F^{n-1}$, normalized on the maximal compact subgroup of $C$.
    
    \item \tdef{$\mu_{C}$} the measure on $C$  corresponding to the standard Haar measure on $F^{\times}\times F^{n-1}$ under the standard identification $C\cong F^{\times}\times F^{n-1}$.    
    \item \tdef{$\Delta$} the discriminant considered as a regular function on $\bfG$.
    \item $\mdef{\bfG^{rss}}\subset \bfG$ the non-vanishing locus of $\Delta$. This is the locus of regular-semi-simple  elements.        
    
    \item $\mdef{G^{rss}}:=\bfG^{rss}(F)$.    
    \item $\mdef{G^{el}}$ the collection of elliptic elements in $G$. i.e. matrices whose characteristic polynomial is separable and irreducible.
    \item For a standard Levi subgroup $M<G$ we denote by $\mdef{M^{el}}$ the collection of elliptic elements in $M$, i.e.  block matrices with each of the blocks being elliptic.
    \item $\mdef{\bfC^{rss}}$ and $\mdef{C^{rss}}$ the images \RamiA{(under $p$)} of $\bfG^{rss}$ and $G^{rss}$ in $\bfC$ and $C$.
    \item $\mdef{p^{rss}}:\bfG^{rss}\to \bfC^{rss}$ the restriction of $p$.
    \item  $\mdef{\Delta_C}$  the discriminant considered as a function on $\bfC$.

    \item $\mdef{C^{el}}$ the image \RamiA{(under $p$)} of $G^{el}$ in $C$. 
    
    \item 
   Similarly $\bfC^{rss}$ \RamiA{(resp. $C^{rss}$)} is  identified with the collection of all separable polynomials in $\bfC$ \RamiA{(resp. $C$)}, and $C^{el}$ with the collection of all irreducible polynomials in $C^{rss}$.
\end{enumerate}

\section{Weak convergence of the averaging to the character - Proof of \Cref{prop:ChartAv}}\label{ssec:pf.ChartAv}
\DimaA{Set $d(\rho)$ to be the formal dimension of $\rho$, as defined in \cite[Theorem 1]{HC_VD} taken w.r.t. the Haar mesure $\mu_{G^{ad}}$. Let $m=m_{v,\varphi}$ be a matrix coefficient of $\rho$.}
\begin{enumerate}[{Case} 1.]
    \item $\rho$ is unitarizable.\\
Let $(\cdot,\cdot)$ be the inner product on $V_\rho$.
 Choose $u\in V_{\rho}$ such that $\varphi(\cdot)=( \cdot, u )$.
Then
$m(y)=(\rho(y)v,u)$. 

We will show that for any $\RamiA{f} \in C_{c}^{\infty}(G)$ we have
$$\langle \cA_{i}(m),f \mu_G \rangle \underset{i\to \infty} \longrightarrow  \frac{m(1)}{d(\rho)}\langle 
 \chi_{\rho},f \mu_G \rangle $$
Notice that
$$\langle \cA_{i}(m),f \mu_G\rangle =\int_{G} f(x) \left (\int_{(G^{ad})_i} m(gxg^{-1})dg \right) dx,$$
where $dg$ is the measure $\mu_{G^{ad}}$.
As $(G^{ad})_i$ is compact we can interchange the order of iterated integration:
\begin{align*}
\langle \cA_{i}(m),f\mu_G\rangle &=\int_{G^{ad}_i}
\left (\int_{G} f(x) m(gxg^{-1})dx\right) dg\\
&=\int_{G^{ad}_i} \left  (\int_{G} f(x)(\rho(gxg^{-1})v,u)dx \right)dg
\end{align*}
By  \cite[Theorem 9]{HC_VD}:
$$(v,u)\langle \chi_{\rho}, f \mu_G\rangle=d(\rho)\int_{G^{ad}}\left(\int_{G} f(\RamiA{x})(\rho(g^{-1}\RamiA{x}g)v,u)d\RamiA{x} \right)dg$$
The result follows.
    \item The general case.\\
    Let $w_\rho$ be the central character of $\rho$. we can write  $w_\rho=w_1w_2$ where $w_1$ is a unitary character of $Z(G)$ and $w_2\RamiA{:=|w_\rho|}$ is a  character of $Z(G)$ that can be extended to a character $w'$ of $G$. Let $\rho_1=(w')^{-1}\rho$. It is easy to see that $\rho_1$ is unitarizable. The assertion follows now from the previous case.
\end{enumerate}
\DimaA{
\begin{remark}
    We use \cite[Theorems 1 and 9]{HC_VD} in the proof. Formally, \cite{HC_VD} assumes characteristic zero, but the proofs of these results do not depend on this assumption.
\end{remark}
}

\section{Norms}\label{sec:Norms}
In this section, we recall basic parts of the theory of norms developed in \cite[\S18]{Kot}.
We will use the following notions from \cite[\S18]{Kot}.
\begin{enumerate}
    \item An \tdef{abstract norm} on a set $Z$ is a real-valued
function $||\cdot ||_Z$ on $Z$ such that $||x||_Z\geq 1$ for all $x \in Z$. 

\item For two abstract norms $||\cdot||_Z^1,||\cdot||_Z^2$ on $Z$ we say that $||x||_Z^1\mdef{\prec}||x||_Z^2$ if there is a constant $c>1$ such that $||x||_Z^1<c(||x||_Z^2)^c$.

    \item We say that two abstract norms $||\cdot||_Z^1,||\cdot||_Z^2$ on $Z$ are equivalent if  $||x||_Z^1\prec||x||_Z^2\prec||x||_Z^1$.
We denote this as 
     $$||\cdot||_Z^1 \mdef{\sim} ||\cdot||_Z^2.$$
    \item Let $\bfM$ be an algebraic variety. In \cite[\S18]{Kot} there is a definition of a canonical equivalence class of abstract norms in $M=\bfM(F)$. The abstract norms in this class are called norms on $M$.    
    \item For a map $\phi:Z_1\to Z_2$ between sets and abstract norms $||\cdot||_{Z_1}, ||\cdot||_{Z_2}$ on these sets, \cite[\S18]{Kot} defines norms $\phi_*(||\cdot||_{Z_1})$ and $\phi^*(||\cdot||_{Z_2})$ on $Im(\phi)$ and $Z_1$  correspondingly by $\mdef[\phi_*(\vert\vert\cdot\vert\vert)]{\phi_*}(||\cdot||_{Z_1})(z):=\inf_{y\in \phi^{-1}(\RamiA{z})}||y||_{Z_1}$, and $\mdef[\phi^*(\vert\vert\cdot\vert\vert)]{\phi^*}(||\cdot||_{Z_2})(y):=||\phi(y)||_{Z_2}$. 
    \item We say that a map $\phi:\bfM\to\bfN$ of algebraic varieties satisfies the \tdef{Norm Descent Property} (in short \tdef{NDP}) if for any two norms $||\cdot||_{M}$ and $||\cdot||_{N}$ on  $M:=\bfM(F)$ and $N:=\bfN(F)$ we have 
    $$(||\cdot||_{N})|_{\phi(M)} \sim \phi_*(||\cdot||_{M}).$$
\end{enumerate}
\begin{notation}\label{not:ov.and.ball}

As a rule, we will put the domain of definition of a norm in a subscript in the notation for that norm.
Given an abstract norm $||\cdot||_X$ on a set $X$, we denote:
\begin{enumerate}
    \item  by $\mdef{ov}_X:X\to \R$ the map given by $ov_X(x)=\log_\ell(||x||_X)$.
\item For $i \in \mathbb{Z}$, $\mdef[\bullet_i]{X_i}:=\{x\in X\, \vert ov_X(x)\leq  i\}.$ Note that this notation might be ambiguous with the notation $K_i$, but we will not consider norms on $K$  so that there is no actual ambiguity.
\end{enumerate}
If we consider more than one norm on the same set $X$, we will distinguish between them and the related notation $ov_X$ and $X_i$ using super-scripts.
\end{notation}

\begin{notation}
    For two sets $Z_1,Z_2$, and abstract norms $||\cdot ||_{Z_1},||\cdot ||_{Z_2}$ on them we define an abstract norm $\mdef{\vert\vert\cdot\vert\vert_{Z_1}\times \vert\vert\cdot\vert\vert_{Z_2}}:=||\cdot||_{Z_1\times {Z_2}}$, where 
    $$||(a,b)||_{{Z_1}\times {Z_2}}:=\max(||a||_{Z_1},||b||_{Z_2}).$$
\end{notation}
We will also use the following facts from \cite[\S 18]{Kot}.

\begin{lemma}[{\cite[Proposition 18.1(1)]{Kot}}]\label{lem:pul.kot}
Given \Dima{a morphism of algebraic varieties} $\phi:\bfM\to\bfN$ and norms $||\cdot||_M$ on $M:=\bfM(F)$ and $||\cdot||_N$ on $N:=\bfN(F)$, we have 
    \begin{enumerate}
        \item\label{lem:pul.kot:1}   $\phi^*(||\cdot||_N)\prec ||\cdot||_M.$
    \item If $\phi$ is a finite map then $\phi^*(||\cdot||_N)\sim ||\cdot||_M.$
    \end{enumerate}
\end{lemma}

\begin{cor}
\label{lem:fin.ndp}
Finite maps satisfy the NDP.
\end{cor}
\begin{lem}[{cf. \cite[18.4]{Kot}}]\label{lem:Kot.comp}
Consider morphisms 
$$
\bfM_1 \xrightarrow{f} \bfM_2  \xrightarrow{g} \bfM_3
$$
of algebraic varieties. 
Assume that the map $f : \bfM_1({F}) \to \bfM_2({F})$ is surjective. Then:
\begin{enumerate}    \item\label{lem:Kot.comp:1} If $f$ and $g$ satisfy the norm descent property, then so does $g \circ f$.
    \item\label{lem:Kot.comp:2} If $g \circ f$ satisfies the norm descent property, then so does $g$.
\end{enumerate}
\end{lem}

\begin{lemma}[{\cite[Theorem 18.2]{Kot}}]\label{lem:kot.sec}
Let $\gamma : \bfM \to \bfN$ be a morphism algebraic varieties. For an open subset $\bfU \subseteq \bfN$, write $\gamma_\bfU$ for the morphism $\gamma^{-1}(\bfU) \to U$ obtained by restriction from $\gamma$.

\begin{enumerate}
    \item The norm descent property for $\gamma: \bfM \to \bfN$ is local
    with respect to the Zariski topology on $\bfN$. In other words, for
    any cover of $\bfN$ by affine open subsets, the morphism $\gamma$ has 
    the norm descent property if and only if the morphisms $\gamma_\bfU$ 
    have the norm descent property for every member $\bfU$ of the open 
    cover.
    
    \item If the morphism $\gamma : \bfM \to \bfN$ admits a section, then $\gamma$ has the norm descent property. More generally, if $\gamma : \bfM \to \bfN$ admits sections locally in the Zariski topology on $\bfN$, then $\gamma$ has the norm descent property.
\end{enumerate}

\end{lemma}

\begin{lemma}[\RamiA{cf.} {\cite[Lemma 18.9]{Kot}}]
    For any torus $\bfS<\bfG$, the map  $\bfG\to \bfG/\bfS$  has the NDP.
\end{lemma}
The following \Rami{two lemmas are} straightforward:
\begin{lem}\label{lem:cart.NDP}
    Let    $$\begin{tikzcd}
\bfM_1   \arrow[r,"f"] \arrow[d] \arrow[dr, phantom, "\square"] & \bfM_2   \arrow[d] \\
\bfM_3  \arrow[r,"g"] & \bfM_4 
\end{tikzcd}
$$
be a Cartesian square of algebraic varieties. If $g$ has NDP then so does $f$.
\end{lem}
\begin{lem}\label{lem:comp.emb.NDP}
Consider morphisms 
$$
\bfM_1 \xrightarrow{f} \bfM_2  \xrightarrow{g} \bfM_3
$$
of algebraic varieties. 
Assume that the map $g : \bfM_2({F}) \to \bfM_3({F})$ is injective. If $f$ and $g$ satisfy the NDP, then so does $g \circ f$.
\end{lem}

We fix the following norms:
\begin{enumerate}
    \item $\mdef[\vert\vert\cdot \vert\vert_F]{||x||_F}=\max(|x|,1)$ 
    \item $\mdef[\vert\vert\cdot \vert\vert_{F^n}]{||\cdot||_{F^k}}:=||\cdot ||_F\times \cdots\times ||\cdot ||_F$. We will use this notation for affine spaces that are equipped with standard basis.

 \item
$\mdef[\vert\vert\cdot \vert\vert_{G}]{||g||_G}=\max\{||g||_{\fg},||\det g||_{\RamiB{F}}^{-1}\},$ \Rami{where $\fg=\fg\fl_n(F)=\mathrm{Mat}_{n\times n}(F)$.}
 \item $\mdef[\vert\vert\cdot \vert\vert_{C}]{||f||_C}=\max(||f||_\fc,||f(0)^{-1}||),$
 \Rami{where $\fc$ is the space of monic polynomials of degree $n$ with coefficients in $F$.}
 \item  $\mdef[\vert\vert\cdot \vert\vert_{G^{ad}}]{||\cdot||_{G^{ad}}}:=ad_*(||\cdot||_{G})$ where $ad:G\to G^{ad}$ is the quotient map.
\item $\mdef[\vert\vert\cdot \vert\vert_{G^{rss}}]{||g||_{G^{rss}}}=\max(||g||_{G},|\Delta(g)^{-1}|)$.
\item $\mdef[\vert\vert\cdot \vert\vert_{C^{ad}}]{||f||_{C^{rss}}}:=\max(||f||_{C},||\Delta(f)^{-1}||).$
\item For any standard torus $A<T$  set  $\mdef{\vert\vert\cdot\vert\vert_{G/A}}:=pr_*(||\cdot||_G)$ where $pr:G\to G/A$ is the projection.
\item For any other variety $\bfV$ we choose a norm $\mdef[\vert\vert\cdot\vert\vert_{V}]{||\cdot||_{V}}$
on $V:=\bfV(F).$  
\end{enumerate}


\section{Some bounds on norms}\label{sec:norm.bnd}

The role of this section is to provide bounds on norms, that enable us to make Harish-Chandra's bounds uniform on the entire group. 

\subsection{Statements and ideas of proof}
In this subsection we formulate some norm bounds - Lemmas \ref{lem:phi} and \ref {lem:Lev.NDP}, and \Cref{prop:norm}. These will be important for \RamiA{the} proof of \Cref{thm:Om.bnd.Av}. 
 


\begin{lemma}\label{lem:phi}
The map $\phi:\bfG\times \bfG^{rss}\to  \bfG^{rss}\times \bfG^{rss}$ 
given by $\mdef \phi(g,h)=(g^{-1}hg,h)$  has the NDP.
\end{lemma}
 We will prove this lemma in \S\ref{ssec:pf.phi}. The proof is based on the notion of companion matrix.
\begin{lemma}\label{lem:Lev.NDP}
    Let $\bfM<\bfG$ be a standard Levi. Then the adjoint action map $\bfG\times (\bfM\cap \bfG^{rss})\to \bfG^{rss}$ has NDP.
\end{lemma}
We will prove this lemma in \S\ref{ssec:pf.Lev.NDP}. The proof is based on the previous lemma and the NDP property of products of polynomials (\Cref{lem:mfin}).

\begin{notation}\label{not:norms.com}
$ $
\begin{itemize} 
\item Define
$\mdef{\bfCom^{rss}}:=\{(x,y)\in \bfG^{rss}\times \bfG \, \vert \, xy=yx\}$. \Rami{Note that it is reduced (in fact even smooth), so we consider it as a variety.}

\item 
For any $x\in G^{rss},$ define an abstract norm $||\cdot ||_{G_x}'$ in the following way. Let $\prod p_i$ be the decomposition of the characteristic polynomial of $x$ into irreducible monic polynomials. Identify $G_x$ with $\prod E_i^{\times},$ where $E_i=F[t]/p_i$. Let $||\cdot||'_{E^{\times}_i}=\max(|\cdot|_{E_i},|\cdot|_{E_i}^{-1})$. 
    Using the identification above, define 
    \begin{equation}       
    \mdef{\vert\vert\cdot \vert\vert'_{G_x}}:=||\cdot||'_{E_1^{\times}}\times \cdots\times ||\cdot||'_{E_k^{\times}}
        \end{equation}

\item     Define $\mdef[\vert\vert\cdot\vert\vert^R_{Com^{rss}}]{||\cdot||^R_{Com^{rss}}}$ by 
        \begin{equation}       
||(x,y)||^R_{Com^{rss}}:=||y||'_{G_x}
    \end{equation}
\item     Define $\mdef{\vert\vert\cdot\vert\vert'_{Com^{rss}}}$ by 
        \begin{equation}       
||\cdot||'_{Com^{rss}}:=\max(pr^*||\cdot||_{G^{rss}},||\cdot||^R_{Com^{rss}}),
    \end{equation}       
    \RamiA{where $pr:Com^{rss}\to G^{rss}$ is the projection on the first coordinate.}
\end{itemize}
\end{notation}

\begin{prop}\label{prop:norm}
    $||\cdot||'_{Com^{rss}}$  is a norm.
\end{prop}
Subsections \ref{ssec:root.bnd}-\ref{ssec:pf.prop.norm} are devoted to the proof of this proposition, the proof itself is in \S \ref{ssec:pf.prop.norm}. Let us first summarize the idea of the proof. The intuitive meaning of the proposition is that given an element $x\in G^{rss}$ and an element $y$ in its centralizer, we can bound (from above and below) the value $||y||'_{G_x}$ in terms of the norm of the pair $(x,y)\in Com^{rss}$. 
The value $||y||'_{G_x}$ can be interpreted as the ``distance" of $y$ from the maximal compact subgroup of the centralizer of $x$.
\begin{enumerate}[Step 1.]
    \item Consider the set $S^{el}:=C^{el}\times P_{n-1}$ where $P_{n-1}$ is the collection of polynomials of degree $\leq n-1$. For each $f\in C^{el}$ we can identify $P_{n-1}$ with the field extension $F[t]/f$. This gives us an abstract norm $||\cdot||'_{S^{el}}$ on $S^{el}$ analogous to the norm $||\cdot||_{Com^{rss}}$, see \Cref{not:S} (\ref{not:S:5}-\ref{not:S:8}) below. We  prove that this abstract norm is equivalent to the restriction of the chosen norm on $S:=C^{rss}\times P_{n-1}$. This step is performed in \Cref{lem:elsim}.

    The intuitive meaning of this lemma is that given a monic irreducible polynomial $f$ and a polynomial $g$ we can bound the norm of $g$ considered as an element in the field $F[t]/f$ in terms of the norm of the pair $(f,g)\in S$.

    The proof of this lemma is based on an inductive argument and on a bound on the minimal distance between roots of a polynomial $f\in C^{rss}$ in terms of its norm in $C^{rss}$, that is given in \Cref{cor:delta.root} below.    
    \item We consider a subset $S^\times \subset S$ consisting of pairs of co-prime polynomials. We also consider the set $S^{el,\times} :=S^\times\cap S^{el}$. We define an abstract norm $||\cdot ||'_{S^{el,\times}}$ analogously to $||\cdot ||'_{S^{el}}$ (see \Cref{not:cross} below)  and prove that it is equivalent to the restriction of the chosen norm on  $S^{\times}$. This step is performed in \Cref{cor:cross}.
    \item Analogously we construct an abstract norm $||\cdot||'_{S^{\times}}$ (see \Cref{not:cross} below) and prove that it is a norm. This step is performed in \Cref{lem:norms.S.cross}.
    \item We deduce the proposition. The relation between the previous step and this proposition is given by a map $\theta:Com^{rss}\to S^{\times}$ defined by $\theta(x,y)=(p(x),g)$ where $g$ is a polynomial such that $y=g(x)$. This step is performed in \S\ref{ssec:pf.prop.norm}.
\end{enumerate}

\subsection{Proof of \Cref{lem:phi}}\label{ssec:pf.phi}
For the proof we will need the following lemma:
\begin{lemma}\label{lem:covering}
    Let $\bfV$ be a vector space and $\bfM$ be an affine variety. Let $f\in\cO(\bfM\times \bfV)$ be  a regular function. Assume that for any $x\in \bfM(\bar F)$ the function $f|_{x\times \bfV(\bar F)}$  is not identically $0$. Then the collection $$\{\bfM^{v}:=\{x|f(x,v)\neq 0\}|v\in \bfV(F)\}$$ covers $\bfM$.
\end{lemma}
\begin{proof}
Follows from the fact that $\bfV(F)$ is Zariski dense in $\bfV$.
\end{proof}
\begin{proof}[Proof of \Cref{lem:phi}]
Let $c:C\to G$ be the map defined by mapping a polynomial to its companion matrix. Let $a:G\times C^{rss}\to G^{rss}$ be the map defined by $a(g,f):=gc(f)g^{-1}$. 
    \begin{enumerate}[Step 1.]
        \item $a$ is onto on the level of $F$-points, and has the norm descent property.\\
        By \Cref{lem:kot.sec},  it is enough to construct a section of $a$, locally in the Zariski topology. For any $v\in F^n$, we say that $A\in G$ is $v$-regular if and only if the matrix $[v,A(v),\dots,A^{n-1}(v)]$
         whose columns are $v,A(v),\dots,A^{n-1}(v)$ is invertible.

         We denote the collection of $v$-regular matrices by \tdef{$\bfG^{r,v}$}.

         Define $\nu_v: \bfG^{r,v}\cap \bfG^{rss}\to \bfG\times \bfC^{rss}$ by 
                 $$\nu_v(A):=([v,A(v),\dots,A^{n-1}(v)],p(A))$$

                 It is easy to see that $\nu_v$ is a section for $a$. Also, by \Cref{lem:covering} the collection  
                 $$\{\bfG^{r,v}\cap \bfG^{rss}|v\in F^n\}$$ covers $\bfG^{rss}$. So we are done by \Cref{lem:kot.sec}.
        \item $\phi$ has the norm descent property.\\
        Let us first explain the idea of the proof. Given two conjugated elements $x_1,x_2\in G^{rss},$ we have to find an element $g\in G$ that conjugates $x_1$ to $x_2$, with an effective bound on its norm. By the previous step, we can find elements $g_1,g_2$ such that $g_i$ conjugates $x_i$ to its companion matrix. Taking the ratio of $g_i$ we get the required element $g$.

        For the formal proof, consider the following diagram.
        $$
\begin{tikzcd}
\bfG\times \bfC^{rss}\times \bfG\times \bfC^{rss} \arrow[r,"a\times a"]   
&
\bfG^{rss}\times \bfG^{rss}
\\
 \bfG\times \bfG\times \bfC^{rss} \arrow[r,"\tilde a"]\arrow[u,"\gamma"] \arrow[ur, phantom, "\square"]   \arrow[d,"Id_\bfG\times a"]& 
 \bfG^{rss} \times_{\bfC^{rss}} \bfG^{rss} \arrow[u,"i"]
 \\
\bfG \times \bfG^{rss} \arrow[rounded corners, to path={ -- ([xshift=32ex]\tikztostart.east) -- 
    node[right]{$\phi$}
    ([xshift=5ex]\tikztotarget.east) -- (\tikztotarget)}]{uur}&
\end{tikzcd}$$
Here, $\gamma$ is given by $\gamma(g_1,g_2,f)=(g_1,f,g_2,f)$ and $\tilde a$ is defined by the Cartesian square.

By the previous step, the map $a\times a$ has NDP. Thus by \Cref{lem:cart.NDP}, the map $\tilde a$  has NDP. Since $i$ is a closed embedding, by \Cref{lem:comp.emb.NDP} this implies that $i\circ \tilde a$ has NDP. So  $\phi\circ Id_G\times a=i\circ \tilde a$ also has NDP. By the previous step $Id_G\times a$ is surjective on the level of points and has NDP. Therefore, by \Cref{lem:Kot.comp}\eqref{lem:Kot.comp:2} we obtain that $\phi$ has NDP as required.
    \end{enumerate}
\end{proof}

\subsection{Bounds on roots of polynomials}\label{ssec:root.bnd}
In this subsection we prove that the product map on polynomials has NDP and give a bound on distances between roots of a polynomial in terms of its norm in $C^{rss}$.

Recall that for any $k\in \Z_{>0}$ we denote by $^k\bfC$  
the variety of monic polynomials of degree $k$. 

\begin{lemma}\label{lem:prod.fin}
For any $k,l\in \Z_{>0}$
the multiplication map 
    $m_{(k,l)}:{}^k\bfC\times {}^l\bfC\to {}^{k+l}\bfC$ is finite. 
\end{lemma}
\begin{proof}
    Consider the map $\mu:({}^1\bfC)^{k+l}\to {}^{k}\bfC\times {}^l\bfC$ defined by 
        \begin{equation}       
    \mu(f_1,\dots,f_{k+l})=(f_1\cdots f_k, f_{k+1}\cdots f_{k+l}). 
        \end{equation}       
It is easy to see that this map is dominant. The composition
    $m_{(k,l)}\circ \mu:{}^1\bfC^{k+l}\to {}^k\bfC\times {}^l\bfC\to {}^{k+l}\bfC$ is finite. Thus \Dima{$m_{(k,l)}$} is finite. 
\end{proof}
By \Cref{lem:fin.ndp} this lemma gives the following corollary.
\begin{cor}\label{cor:prod.NDP}
For any $k\in \Z_{>0}$, the multiplication map 
    $m_{(k,l)}:{}^k\bfC\times {}^l\bfC\to {}^{k+l}\bfC$ satisfies the norm descent property. 
\end{cor}

\begin{notation} $ $
    \begin{itemize}
        \item For $f\in C$, denote 
    $\mdef[\vert\vert\cdot\vert\vert_C^{root}]{||f||_C^{root}}=\max\{||\lam||_{\Rami{\bar F}} \, \vert \, \lam\in \bar F \text{ with }f(\lam)=0\}.$ 
    \item  For $f\in C^{rss}$, denote 
        \begin{equation}       
    \mdef[\vert\vert\cdot\vert\vert_C^{\delta-root}]{||f||_{C^{rss}}^{\delta-root}}=\max\left\{\left|\left|\frac 1{\lam-\mu}\right|\right|_{\Rami{\bar F}} \, : \, \lam,\mu\in \bar F \text{ with }f(\lam)=f(\mu)=0 \text{ and } \lam\neq \mu\right\}.
        \end{equation}
    \end{itemize}
\end{notation}

\begin{lem}\label{lem:root}
 We have $||\cdot ||^{root}_{C}\prec ||\cdot ||_{C}$.
\end{lem}
\begin{proof}
The classical bounds of Cauchy and Lagrange give 
$||\cdot ||^{root}_{C}\leq ||\cdot ||_{C}$. This implies the assertion.
\end{proof}

\begin{cor}\label{cor:delta.root}
We have $||\cdot ||^{\delta-root}_{C^{rss}}\prec ||\cdot ||_{C^{rss}}$.
\end{cor}
\begin{proof}
Let $f\in C^{rss}$  and let $\lambda_i$ be the roots of $f$ in $\bar F$.
We have 
   \begin{equation}       
|\Delta(f)|_F=\prod_{1\leq i<j\leq n}|\lambda_i-\lambda_j|_{\bar F}.
    \end{equation}       

Thus, for any $i\neq j$ we have 
    \begin{equation}       
|\lambda_i-\lambda_j|_{\bar F}\geq |\Delta(f)|_F (||f||_{C}^{root})^{-n^2} \geq (||f||_{C^{rss}})^{-1}(||f||_{C}^{root})^{-n^2} 
    \end{equation}       

By the previous lemma (\Cref{lem:root}), this implies the assertion.
\end{proof}

\subsection{Norms on $S^{el}$}
In this subsection we define an abstract norm $||\cdot||'_{S^{el}}$ and prove that it is a norm.
\begin{notation}\label{not:S} Let $k<n$ be an integer. We introduce the following notation.
    \begin{enumerate}
    \item $\mdef{\bfP_k}$ - the variety of polynomials of degree $\leq k$.
    \item $\mdef{\bfS_k}:=\bfC^{rss} \times \bfP_k$, $S_k:=\bfS_k(F).$
\item $\mdef{S^{el}_k}:=(C^{el} \times P_k)$. 
\item  $\mdef{\bfS}:=\bfS_{n-1}$, $\mdef{S}:=S_{n-1}$, $\mdef{S^{el}}:=S^{el}_{n-1}$. 
\item\label{not:S:5}  For any $f\in C^{rss},$ define an abstract norm $\mdef{\vert\vert\cdot \vert\vert_{(F[t]/f)}'}$ in the following way. 
Let $f=\prod f_i$ be the decomposition of $f$ into irreducible monic polynomials. Identify $F[t]/f$ with $\prod E_i,$ where $E_i:=F[t]/f_i$. 
    Using this identification, define 
        \begin{equation*}       
||\cdot ||'_{(F[t]/f)}:=||\cdot||_{E_1}\times \cdots\times ||\cdot||_{E_d}
    \end{equation*}
\item\label{not:S:6}
$\mdef[\vert\vert\cdot\vert\vert^R_{S_k}]{||(f,g)||^R_{S_k}}:=||(g \mo f)||'_{(F[t]/f)}$

\item\label{not:S:7} $\mdef{\vert\vert\cdot \vert\vert'_{S_k}}:=\max(pr^*(||\cdot||_{C^{rss}}),||(f,g)||_{S_k}^R)$, where $pr:S_k \to C^{rss}$ is the projection. 
\item \label{not:S:8}
$\mdef{\vert\vert\cdot\vert\vert^R_{S^{el}}}:=(||\cdot||^R_{S})|_{S^{el}}$;
$\mdef{\vert\vert\cdot\vert\vert'_{S^{el}}}:=(||\cdot||'_{S})|_{S^{el}}$;
$\mdef{\vert\vert\cdot\vert\vert_{S^{el}}}:=(||\cdot||_{S})|_{S^{el}}$
\end{enumerate}
\end{notation}
\begin{lem}\label{lem:elsim} For any $k<n$ we have
    $||\cdot||'_{S_k^{el}}\sim ||\cdot||_{S_k^{el}}$
\end{lem}
Before giving the formal proof let us indicate its idea: The main part is the inequality $||\cdot||_{S_k^{el}}\prec ||\cdot||'_{S_k^{el}}$.  Its intuitive meaning is that for a pair $(f,g)\in S^{el}$ we can bound (from above) the coefficients of $g$ by the norm of $f\in C^{rss}$ and the norm of $g(t)$ considered as an element in $F[t]/f$. The proof is by induction on the degree of $g$. The main step is to bound the norm of the leading coefficient of $g$. This is done in Steps \ref{pf:elsim:1},\ref{pf:elsim:2} below. The bound on the rest of the coefficients follows by induction (see Steps \ref{pf:elsim:3},\ref{pf:elsim:4} below). 

The proof of the bound on the leading coefficient of $g$ is based on a relation between the norm of this leading coefficient, the norm of $g(t)$ and the distances between the roots of $f$ and of $g$ (see \eqref{pf:elsim:eq1} below). This relation implies that if the leading coefficient is too large, then one of the roots of $f$ is very close to one of the roots of $g$. Using the Galois action, we deduce that any root of $f$ is very close to some root of $g$. Since we can bound the distance  between roots of $f$ (\Cref{cor:delta.root}), this contradicts the fact that $\deg(g)<\deg(f)$.
\begin{proof}[Proof of \Cref{lem:elsim} ]
We first show $||\cdot||'_{S_k^{el}}\prec ||\cdot||_{S_k^{el}}$. 
    Let $(f,g)\in S_k^{el}$. We have 
        \begin{equation}       
|g \,\mathrm{mod}\, f|_{F[t]/f}=|res(f,g)|_F,
    \end{equation}       
where $res(f,g)$ is the resultant of $f$ and $g$.  Consider $res$ as a map $S_k^{el}\to F$. We obtain 
    \begin{equation}\label{eq.R.el}
        ||\cdot ||^R_{S_k^{el}}=res^*(||\cdot||_{F}).
    \end{equation}
    
    Thus by \Cref{lem:pul.kot} we have $||\cdot||^R_{S_k^{el}}\prec ||\cdot||_{S_k^{el}}$ 
    and thus        
    $||\cdot||'_{S_k^{el}}\prec ||\cdot||_{S_k^{el}}$. 

    It remains to show that $||\cdot||_{S_k^{el}}\prec ||\cdot||'_{S_k^{el}}$.
We will prove this by induction on $k$. 
Define 
\begin{itemize}
    \item $T:S^{el}_k\to F$ by     \begin{equation}       
T\left(\left(f,\sum_{i=0}^{k}a_it^i\right)\right)=a_k
    \end{equation}
    \item $H:S^{el}_k\to S^{el}_{k-1}$ by     \begin{equation}       
H\left(\left(f,\sum_{i=0}^{k}a_it^i\right)\right)=\left(f,\sum_{i=0}^{k-1}a_it^i\right)
    \end{equation}
\item $||\cdot||''_{S^{el}_k}:=\max(T^*(||\cdot||_{F}),H^*(||\cdot||'_{S^{el}_{k-1}}))$. 
\end{itemize}

Note that 
    \begin{equation}       
||\cdot||_{S^{el}_k}:=\max(T^*(||\cdot||_{F}),H^*(||\cdot||_{S^{el}_{k-1}})).
    \end{equation}
Thus, the induction hypothesis implies that 
$||\cdot ||_{S^{el}_k}  \prec ||\cdot ||''_{S^{el}_k}$.
Thus it is enough to show that $||\cdot ||''_{S^{el}_k} \prec ||\cdot ||'_{S^{el}_k}$. 

Denote:
    \begin{equation}       
S^{el,1}_k=\left\{(f,g)\in S^{el}_k :  \left| g \mo f\right|_{F[t]/f}\leq 1 \right\} 
    \end{equation}       

We divide the proof into the following steps.

\begin{enumerate}[Step 1.]
    \item\label{pf:elsim:1} $(T^*(||\cdot ||_{F})|_{S^{el,1}_k}\prec ||\cdot ||_{S^{el}_k}'|_{S^{el,1}_k}$\\
    Let 
        \begin{equation}       
    (f=\sum_{i=0}^n b_it^i,g=\sum_{i=0}^k a_it^i)\in S^{el,1}_k
        \end{equation}       
    By \Cref{cor:delta.root}, it is enough to show that 
        \begin{equation}       
|a_k|_F\leq     (||f||_{C^{rss}}^{\delta-root})^n
    \ell^{n}.
        \end{equation}
    Assume the contrary.
    Let $E=F[t]/f$ and 
    let $E'/F$ be the normal closure of $E/F$.
 Choose a finite field extension $L\Rami{'}/E'$ such that $g$ splits in $L\Rami{'}$. \Rami{Let $L/F$ be the normal closure of $L'/F$. By \cite[\href{https://stacks.math.columbia.edu/tag/0BME}{Lemma 0BME}]{SP}} the map $Aut(L/F)\to Gal(E'/F)$ is onto.

  Let $\lambda_i$ be the roots of $f$ and $\mu_j$ be the roots of $g$ in $L$. Consider $g(t)$ as an element in $E$.  
  We have 
  \begin{align}\label{pf:elsim:eq1}
    1&\geq |g(t)|_E=|a_k\prod_{\Rami{j=1}}^{\Rami{k}} (t-\mu_j)|_E
    =|a_k|_F\cdot |\prod_{\Rami{j=1}}^{\Rami{k}} \Rami{\prod_{i=1}^n} (\lambda_i-\mu_j)|_L^{\frac 1n}>\\&>(||f||_{C^{rss}}^{\delta-root})^{n}\ell^{n}\prod_{\Rami{j=1}}^{\Rami{k}} \Rami{\prod_{i=1}^n} |\lambda_i-\mu_j|_L^{\frac 1n}.   
  \end{align}
  So there exists \Rami{$(i,j)\in 
  \{1,\dots, n\}\times \{1,\dots, k\}$} such that     \begin{equation}       
|\mu_\Rami{j}-\lambda_\Rami{i}|_E<
  ||f||_{C^{rss}}^{\delta-root}\ell^{-1}.
      \end{equation}       
\Rami{Fix such $(i,j)$.}

On the other hand 
for any \Rami{$i_1\neq i_2\in \{1,\dots, n\}$} we have
$$|\lambda_{\Rami{i_1}}-\lambda_{\Rami{i_2}}|_E>
  ||f||_{C^{rss}}^{\delta-root}\ell^{-1}.$$  
For any $l\in\{1,\dots,n\}$, let $\gamma_l\in Aut(L/F)$ be such that $\gamma_l\lambda_{\Rami{i}}=\lambda_l$.
Set $\mu_l:=\gamma_l \mu_{\Rami j}$.
We have:
$$|\lambda_l-\mu_l|_L=|\lambda_j-\mu_i|_L<\min_{\Rami{i_1,i_2}}|\lambda_\Rami{i_1}- \lambda_\Rami{i_2}|$$
So $\mu_l$ are all distinct \Rami{when $l$ ranges over $\{1,\dots,n\}$}. On the other hand $\mu_l$ are roots of $g$. So $\deg(g)\geq n$. Contradiction.
    \item\label{pf:elsim:2} $T^*(||\cdot ||_{F})\prec ||\cdot ||_{S_k}'$\\
    By the previous step, we know that there exists $c>1$ such that for every $(f,g)\in S_k^{el,1}$ we have 
    $||a_k||_F<c(||(f,g)||'_{S_k})^c$.

    Fix $(f,g=\sum_{i=0}^ka_it^k)\in S_k^{el}$. 
    It  is enough to show that 
    $$||a_k||_F<c(||(f,g)||'_{A_k'})^{c+n}.$$ Let $M:=\left|\sum a_it^i\right|_{F[t]/f}$.
    We may assume that $M> 1$. Find $b\in F$ such that $M^n\geq |b| \geq M$.  We have
    $$\left|\frac{a_k}{b}\right|_F\leq \left|\left|\frac{a_k}{b}\right|\right|_F<c\left(\left|\left|\left(f,\frac{g}{b}\right)\right|\right|'_{S_k^{el}}\right)^c .$$
    Thus 
    \begin{align*}
    |a_k|&<|b|c\left(\left|\left|\left(f,\frac{g}{b}\right)\right|\right|_{S^{el}_k}'\right)^c\\ 
    &\leq M^nc\left(\left|\left|\left(f,\frac{g}{b}\right)\right|\right|_{S^{el}_k}'\right)^c \leq c\left(\left|\left|\left(f,g\right)\right|\right|_{S^{el}_k}'\right)^{c+n}
    \end{align*}
Thus $$||T((f,g))||=||a_k||_F<c(||(f,g)||'_{S^{el}_k})^{c+n}.$$
    \item \label{pf:elsim:3}
    $H^*\left(||\cdot ||'_{S^{el}_{k-1}}\right)\prec \max\left((||\cdot ||'_{S^{el}_{k}},T^*\left(||\cdot ||_{F}\right)\right)$.\\
    Let $y:=\left(f=\sum_{i=0}^n b_it^i,g=\sum_{i=0}^k a_it^i\right)\in S_k^{el}$. We have 
    \begin{align*}        
    \left|\sum_{i=0}^{k-1}a_it^i\right|_E&=\left|\sum_{i=0}^{k}a_it^i-a_kt^k\right|_E
    \leq \left|\sum_{i=0}^{k}a_it^i\right|_E+\left|a_kt^k\right|_E
    \\&\leq    ||y||_{S_k^{el}}'+||T(y)||_{S_k^{el}}\cdot (|t|_E)^k
    \\&=    ||y||_{S_k^{el}}'+||T(y)||_{S_k^{el}}\cdot |b_0|_F^{k/n} \\&\leq||y||_{S_k^{el}}'+||T(y)||_{S_k^{el}}\cdot ||y||'_{S_k^{el}}
    \\&\leq  2\max(||y||'_{S_k^{el}},||T(y)||_{S_k^{el}})^2
    \end{align*}
    Thus $||H(y)||'_{S_k^{el}} < 2\max(||y||'_{S_k^{el}},||T(y)||_{S_k^{el}})^2$.
    \item\label{pf:elsim:4} $||\cdot ||''_{S_k^{el}} \prec ||\cdot ||'_{S_k^{el}}$. \\
    \begin{align*}
        ||\cdot ||_{S_k^{el}}''&=\max\left(H^*\left(||\cdot||'_{S_{k-1}^{el}}\right),T^*\left(||\cdot ||_F\right)\right)\\
        &\prec\max(\max(||\cdot||_{S_{k}^{el}}',T^*\left(||\cdot ||_F\right)),T^*\left(||\cdot ||_F\right))\\
        &\prec         \max(\max(||\cdot||'_{S_{k}^{el}},||\cdot||'_{S_{k}^{el}}),||\cdot||'_{S_{k}^{el}})=||\cdot||'_{S_{k}^{el}}
    \end{align*}
\end{enumerate}
\end{proof}

\subsection{Norms on $S^{el,\times}$}
Let $\mdef{\bfS^{\times}}\subset \bfS$ denote the collection of pairs of coprime polynomials, and $\mdef{S^{\times}}:=\bfS^{\times}(F)$.
In this subsection we define an abstract norm ${\vert\vert\cdot\vert\vert'_{S^{el,\times}}}$  on $\mdef{S^{el,\times}}:=S^{el}\cap S^{\times}$ 
and prove that it is equivalent to  $||\cdot||_{S^{el,\times}}:=||\cdot||_{S^{\times}}|_{S^{el,\times}}$.

\begin{notation} \label{not:cross}
Define:
$ $
    \begin{enumerate}
\item For any $f\in C^{rss},$ define an abstract norm $||\cdot ||_{(F[t]/f)^{\times}}'$ in the following way. 
Let $f=\prod f_i$ be the decomposition of $f$ into irreducible monic polynomials. Identify $(F[t]/f)^{\times}$ with $\prod E_i^{\times},$ where $E_i:=F[t]/f_i$. 
    Using this identification, define 
\begin{equation*}       
\gendef{||\cdot ||'_{(F[t]/f)^{\times}}}:=||\cdot||_{E_1^{\times}}\times \cdots\times ||\cdot||_{E_k^{\times}}
    \end{equation*}

\item 
    $\mdef[\vert\vert(\cdot,\cdot)\vert\vert^R_{S^{\times}}]{\vert\vert(f,g)\vert\vert^R_{S^{\times}}}:=||(g \mo f)||'_{(F[t]/f)^{\times}}$
\item  $\mdef{\vert\vert\cdot \vert\vert'_{S^{\times}}}:=\max(pr_\times^*(||\cdot||_{C^{rss}}),||(f,g)||_{S^{\times}}^R)$,  where $pr_\times:S^\times\to C^{rss}$  is the projection.
\item $\mdef{\vert\vert\cdot\vert\vert^R_{S^{el,\times}}}:=(||\cdot||^R_{S^{\times}})|_{S^{el,\times}}$; $\mdef{\vert\vert\cdot\vert\vert'_{S^{el,\times}}}:=(\vert\vert\cdot\vert\vert'_{S^{\times}})|_{S^{el,\times}}$;\\ 
$\mdef{\vert\vert\cdot\vert\vert_{S^{el,\times}}}:=(\vert\vert\cdot\vert\vert_{S^{\times}})|_{S^{el,\times}}$;
\end{enumerate}
\end{notation}

\begin{lem}\label{cor:cross}
We have $||\cdot||'_{S^{el,\times}}\sim ||\cdot||_{S^{el,\times}}$    
\end{lem}
\begin{proof}
As in \eqref{eq.R.el} we have: 
    \begin{equation}\label{eq.R.el.times}
        ||\cdot ||^R_{S^{el,\times}}=res^*(||\cdot||_{F^\times}).
    \end{equation}
Therefore
    $$||\cdot||'_{S^{el,\times}} =\max(||\cdot||'_{S^{el}}, res^*(||\cdot||_{F^\times})).$$
On the other hand, we have:
    $$||\cdot||_{S^{el,\times}} \sim\max(||\cdot||_{S^{el}}, res^*(||\cdot||_{F^\times}))$$

    The lemma follows now from \Cref{lem:elsim}.
\end{proof}

\subsection{Norms on $S^{\times}$}
In this subsection we prove:
\begin{lemma}\label{lem:norms.S.cross}
$||\cdot||_{S^{\times}}\sim||\cdot||'_{S^{\times}}$.
\end{lemma}
Let us first explain the idea of the proof.
The proof is by induction on $n$.

Let $\lambda=(k_1,\dots,k_r)$ be a composition of $n$, namely $n=k_1+\dots+k_r$. Recall our convention of using $k_i$ and $\lambda$ as left super-script (\S\ref{ssec:conv}(\ref{ssec:conv:5}-\ref{ssec:conv:6})).

The proof is  based on the construction of  the following Cartesian squares:
$$
\begin{tikzcd}
  {}^{\lambda}S^{\times} \arrow[d]\ar[dr, phantom, "\square"]  & {}^{\lambda}S^{cop,\times} \ar[dr, phantom, "\square"]\arrow[d] \arrow[l] \arrow[r] & S^{\times} \arrow[d]  \\
  {}^{\lambda}C^{rss} & {}^{\lambda}C^{cop,rss} \arrow[l] \arrow[r]  & C^{rss} 
\end{tikzcd}
$$
In this diagram  the left horizontal arrows are open embeddings, the lower right horizontal arrow is coming from polynomial multiplication and the upper right horizontal arrow is based on polynomial multiplication on the $C$-coordinates and on the Chinese remainder theorem on the $P$-coordinates.

We prove that the right horizontal arrows have the NDP. Based on this we show that,
for a proper composition $\lambda$, the induction hypothesis implies the required  equivalence on the image of the upper right horizontal arrow. When we take the union of these images ranging over all such $\lambda$ we are left with $S^{el,\times}$. So the lemma will follow from \Cref{cor:cross}.

Throughout this subsection $\lambda=(k_1,\dots,k_r)$ denotes a composition of $n$.
\begin{notation} 
We introduce the following notation.    
    \begin{enumerate}
    \item $\mdef{^{\lambda}\bfC^{cop}}\sub {}^{\lambda}\bfC$ - the open subset of tuples consisting of pairwise co-prime polynomials.
    \item$\mdef{^{\lambda}\bfC^{rss,cop}}:={}^{\lambda}\bfC^{rss}\cap {}^{\lambda}\bfC^{cop}$
    \item $^{\lambda}\bfS^{cop} \subset {}^{\lambda}\bfS$ by $\mdef{^{\lambda}\bfS^{cop}} :\cong {}^{\lambda}\bfC^{cop,rss} \times \bfP_{k_1-1} \times \cdots \bfP_{k_r-1}$ under the identification $$^{\lambda}\bfS\cong {}^{\lambda}\bfC^{rss} \times \bfP_{k_1-1} \times \cdots \bfP_{k_r-1}.$$
    \item Similarly we define $\mdef{^{\lambda}\bfS^{\times,cop}}$.
\end{enumerate}
\end{notation}

We now define the notions of division with residue and modular inversions of polynomials as a morphism of algebraic varieties. 

\begin{definition}$\,$
\begin{itemize}
    \item Define $\mdef{mod_{l,k}}:{}^k\bfC\times \bfP_l\to \bfP_{k-1}$ in the following way: for any unital commutative ring $A$, and any $f\in {}^k\bfC(A), g \in \bfP_l(A)$ we set   $mod(f,g)\in \bfP_{k-1}(A)$ to be the unique element such that 
    $$g \equiv mod(f,g) \mod f.$$ 
    \item Define $\mdef{inv_{l,k}}:{}^{(l,k)}\bfC^{cop}\to \bfP_{k-1}$  in the following way: for any unital commutative ring $A$, and any $(f,g)\in {}^{(l,k)}\bfC^{cop}(A)$ we set   $inv_{l,k}(f,g)\in \bfP_{k-1}(A)$ to be the unique element such that $$inv(f,g)f \equiv 1 \mod g.$$ 
\end{itemize}
     
\end{definition}

\begin{notation}
$ $
\begin{enumerate}
        \item Define $mod_{\lambda}:{}^{\lambda}\bfC^{cop,rss} \times \bfP_{n-1}\to {}^{\lambda}\bfS^{cop}$ by     $$\mdef{mod_{\lambda}}((f_1,\dots,f_r),g):=((f_1,mod_{n-1,k_1}(g,f_1) )\dots,(f_r, mod_{n-1,k_r}(g,f_r))).$$
    \item $\mdef{m^{C^{rss}}_{\lambda}}: {}^{\lambda}\bfC^{cop,rss}\to \bfC^{rss}$ to be the map given by the product of polynomials.
    \item $\mdef{m_{\lambda}^{S}}:=m^{C^{rss}}_{\lambda}\times Id_{\bfP_{n-1}}: {}^{\lambda}\bfC^{cop,rss}\times \bfP_{n-1}\to \bfC^{rss}\times \bfP_{n-1}=:\bfS$
\end{enumerate}
\end{notation}
To sum up we have the following diagram:
$$
\begin{tikzcd}
{{}^{\lambda}\bfS^{cop}} &[3em] {{}^{\lambda}\bfC^{cop,rss}\times \bfP_{n-1}}\arrow[d] \arrow[r, "m^{S}_{\lambda}"] \arrow[l,"mod_{\lambda}",swap] \ar[dr, phantom, "\square"] &[3em] \mathbf S\arrow[d]             \\
& {}^{\lambda}\bfC^{cop,rss}  \arrow[r,, "m_{\lambda}^{C^{rss}}",swap]                        &[5em] \bfC^{rss} 
\end{tikzcd}
$$


\begin{lemma}[Relative version of the Chinese remainder theorem]    
$mod_{\lambda}$ is an  isomorphism. 
\end{lemma}
\begin{proof}
    The proof follows the proof of the Chinese remainder theorem.
    By induction, it is enough to prove the lemma for the case when $\lambda=(k_1,k_2)$ is of length 2. In this case, we can define the inverse morphism by the following formula. 
    $$ ((f_1,g_1),(f_2,g_2)) \mapsto ((f_1,f_2),mod_{2(k_1+k_2),n}(g_1f_2inv_{k_2,k_1}(f_2,f_1)+g_2f_1inv_{k_1,k_2}(f_1,f_2),g_2)) $$ 
\end{proof}

\begin{lemma}\label{lem:mfin}
    $m^{C^{rss}}_{\lambda}$ is a finite  map, in particular it has the NDP property.
\end{lemma}
\begin{proof}
By induction, it is enough to show this for $\lambda$ of length 2. 
This follows from \Cref{lem:prod.fin}.
\end{proof}

\begin{cor}\label{cor:mfin}
    $m^S_{\lambda}$ is a finite  map.
\end{cor}

\begin{notation}$\,$
\begin{enumerate}
    \item Define $\mdef{chin_{\lambda}}:=m_{\lambda} \circ mod_{\lambda}^{-1}:{{}^{\lambda}\bfS^{cop}}\to \bfS$
    \item Define $\mdef{chin^{\times}_{\lambda}}:{{}^{\lambda}\bfS^{cop},\times}\to \bfS^{\times}$ to be the restriction of $chin_{\lambda}$.
\end{enumerate}
\end{notation}
    
From \Cref{cor:mfin} we obtain 
\begin{cor}\label{cor:chin.ndp}
    The maps $chin_{\lambda}$ and $chin^{\times}_{\lambda}$ are finite and thus satisfy NDP. 
\end{cor}

\begin{notation}$ $
    \begin{enumerate}
        \item  Denote $\mdef{\vert\vert\cdot \vert\vert^R_{{}^{\lambda} S^{\times}}}$, and $\mdef{\vert\vert\cdot \vert\vert^1_{{}^\lambda S^{\times}}}$ to be the natural analogues of $||\cdot ||^R_{S^{\times}}$ and $||\cdot ||^1_{S^{\times}}$. 
        \item Denote 
        $\mdef{\vert\vert\cdot \vert\vert^R_{{}^\lambda S^{cop,\times}}}:=(||\cdot ||^I_{{}^\lambda S^{\times}}|)_{{}^\lambda S^{cop,\times}}$ and 
        \begin{multline*}
        \mdef{\vert\vert\cdot \vert\vert^1_{{}^\lambda S^{cop,\times}}}:=\max(||\cdot ||^R_{{}^\lambda S^{cop,\times}}, pr_{{}^\lambda C^{rss,cop}}^*
        (||\cdot||_{{}^\lambda C^{rss,cop}}))=\\
        \max((||\cdot ||^1_{{}^\lambda S^{}})|_{{}^\lambda S^{cop,\times}}, pr_{{}^\lambda C^{rss,cop}}^*
        (||\cdot||_{{}^\lambda C^{rss,cop}}))
        \end{multline*}
        where 
        $pr_{{}^\lambda C^{rss,cop}}:{}^\lambda S^{cop,\times}\to {}^\lambda C^{rss,cop}$ is the 
        projection.
    \end{enumerate}
\end{notation}

\begin{proof}[Proof of \Cref{lem:norms.S.cross}.]
We prove the Lemma by induction. Thus from now on we assume that it holds for any smaller value of $n$. Let $\lambda$ be a proper composition of $n$.
    \begin{enumerate}[Step 1.]
        \item $||\cdot ||'_{{}^\lambda S^{\times}} \sim ||\cdot ||_{{}^\lambda S^{\times}}$.\\ 
        Follows immediately from the induction hypothesis.
        \item $||\cdot ||'_{{}^\lambda S^{cop,\times}} \sim ||\cdot ||_{{}^\lambda S^{cop,\times}}$.\\ 
        Follows immediately from the previous step.
        \item $(||\cdot ||'_{S^{\times}})|_{chin_\lambda^{\times}({}^\lambda S^{cop,\times})} \sim (||\cdot ||_{S^{\times}})_{chin_\lambda^{\times}({}^\lambda S^{cop,\times})}.$\\

        Consider the following Cartesian square:
        $$
\begin{tikzcd} {}^{\lambda}S^{cop,\times}\arrow[d,"\Rami{pr_{cop}}"] \arrow[r, "chin^{\times}_{\lambda}"]  \ar[dr, phantom, "\square"] & S^{\times}\arrow[d,"\Rami{pr}"]             \\
 {}^{\lambda}C^{cop,rss}  \arrow[r,"m_{\lambda}^{C^{rss}}",swap]                        & C^{rss} 
\end{tikzcd}
$$
Here, $\Rami{pr_{cop}}$ and $\Rami{pr}$ are the projections \Rami{on the first coordinates}.
        By \Cref{cor:chin.ndp}, $chin_\lambda^{\times}$ has NDP and hence   
        $$(||\cdot ||_{S^{\times}})|_{chin_\lambda^{\times}({}^\lambda S^{cop,\times})}\sim  
        (chin_\lambda^{\times})_*(||\cdot ||_{{}^\lambda S^{cop,\times}}).$$
        By definition 
$$(||\cdot ||^R_{S^{\times}})|_{chin_\lambda^{\times}({}^\lambda S^{\times})}\sim  
        (chin_\lambda^{\times})_*(||\cdot ||^R_{{}^\lambda S^{cop,\times}}).$$ 
        By \Cref{lem:mfin}, 
         $$(||\cdot ||_{C^{rss}})|_{m^{C^{rss}}_\lambda({}^{\lambda}C^{rss})}\sim  
        (m^{C^{rss}}_\lambda)_*(||\cdot ||_{{}^\lambda C^{cop,rss}}).$$
        Therefore, 
        $$(\Rami{pr}^{*}(||\cdot ||_{C^{rss}}))|_{chin_\lambda^{\times}({}^\lambda S^{\times})}\sim  
        (chin^{\times}_\lambda)_*(\Rami{pr_{cop}}^*(||\cdot ||_{{}^\lambda C^{cop,rss}})).$$
So we obtain 
\begin{multline*}
    (||\cdot ||_{S^{\times}})|_{chin_\lambda^{\times}({}^\lambda S^{cop,\times})}\sim  
        (chin_\lambda^{\times})_*(||\cdot ||_{{}^\lambda S^{cop,\times}})\sim
        (chin_\lambda^{\times})_*(||\cdot ||'_{{}^\lambda S^{cop,\times}})=\\=
        (chin_\lambda^{\times})_*(\max(||\cdot ||^R_{{}^\lambda S^{cop,\times}}, \Rami{pr_{cop}}^*(
        ||\cdot ||_{{}^\lambda C^{cop,rss}})))=\\
        \max\left ((chin_\lambda^{\times})_*(||\cdot ||^R_{{}^\lambda S^{cop,\times}}), (chin_\lambda^{\times})_*\Rami{pr_{cop}}^*(
        ||\cdot ||_{{}^\lambda C^{cop,rss}})\right )\sim\\
        \max\left((||\cdot ||^R_{S^{\times}})|_{chin_\lambda^{\times}({}^\lambda S^{\times})},  (\Rami{pr}^{*}(||\cdot ||_{C^{rss}}))|_{chin_\lambda^{\times}({}^\lambda S^{\times})}\right)= \\
        \max\left(||\cdot ||^R_{S^{\times}},  \Rami{pr}^{*}(||\cdot ||_{C^{rss}})\right)|_{chin_\lambda^{\times}({}^\lambda S^{\times})}=\left(||\cdot ||'_{S^{\times}}\right)|_{chin_\lambda^{\times}({}^\lambda S^{\times})}\\
\end{multline*}
\item $||\cdot||_{S}\sim||\cdot||'_{S}$\\
This follows from the previous step, \Cref{lem:elsim} and the following equality
$$ S^{\times}=\left( \bigcup_{\lambda}{}^{\lambda}S^{\times}\right) \cup S^{el,\times},$$
where $\lambda$ ranges over all proper compositions of $n$.
    \end{enumerate}
\end{proof}
\subsection{Norms on $Com^{rss}$ and the proof of \Cref{prop:norm}}\label{ssec:pf.prop.norm}
The next lemma follows from standard linear algebra and basic algebro-geometric considerations.
\begin{lem}
    Let $\bfX$ be an algebraic variety and $V$ be a vector space over $F$. Let $\psi,\phi_i:\bfX\to V$ for $i=1,\dots,k$ be morphisms of algebraic varieties. 
    Assume that for any $x\in \bfX$ there exist unique $a_i\in F$ such that $\sum a_i\phi_i(x)=\psi(x)$. Then there exist regular functions $A_i\in \cO_{\bfX}(\bfX)$ such that $\sum A_i\phi_i=\psi$. 
\end{lem}

\begin{cor}\label{cor:LinAlg}
    There exists a morphism $\xi:\bfCom^{rss} \to \bfP_n$ such that for any ring $A$ and any $(g_1,g_2)\in \bfCom^{rss}(A)$ we have 
    $\xi((g_1,g_2))(g_1)=g_2$.
\end{cor}

\begin{proof}[Proof of \Cref{prop:norm}]
\Rami{Let $\xi$ be as in \Cref{cor:LinAlg}. Define $\theta:\bf{Com}^{rss}\to \bfS^\times$ by $$\theta(g_1,g_2)=(p(g_1),\xi((g_1,g_2))).$$
By \Cref{cor:LinAlg}
}
   we have the \Rami{following} Cartesian square:
    $$\begin{tikzcd}
 \bfCom^{rss}\arrow[d,"\Rami{pr_G}"] \arrow[r, "\theta"]  \ar[dr, phantom, "\square"] & \bfS^{\times}\arrow[d,"\Rami{pr_C}"]             \\
 \bfG^{rss}  \arrow[r,"p|_{\bfG^{rss}}",swap]                        & \bfC^{rss} 
\end{tikzcd}$$
\Rami{Here $pr_G, pr_C$ are the projections on the first coordinates. }
So by \Cref{lem:norms.S.cross},
\Dima{
\begin{equation}\label{prop:norm:eq1}
\max(\theta^*(||\cdot||_{S^{\times}}),\Rami{pr_G^*}(||\cdot||_{G^{rss}}))\sim
    \max(\theta^*(||\cdot||'_{S^{\times}}),\Rami{pr_G^*}(||\cdot||_{G^{rss}})).
\end{equation}
    Thus:}
\begin{align*}
    ||\cdot ||_{Com^{rss}}&\sim  \max(\theta^*(||\cdot||_{S^{\times}}),\Rami{pr_G^*}(||\cdot||_{G^{rss}}))\overset{\eqref{prop:norm:eq1}}{\sim}
    \max(\theta^*(||\cdot||'_{S^{\times}}),\Rami{pr_G^*}(||\cdot||_{G^{rss}}))\sim
    \\&\sim
    \max(\theta^*(\max(\Rami{pr_C^*}(||\cdot||_{C^{rss}}),||\cdot||^R_{S^{\times}} )),\Rami{pr_G^*}(||\cdot||_{G^{rss}}))=
    \\&=    \max(\theta^*\left(\Rami{pr_C^*}(||\cdot||_{C^{rss}})\right), \theta^*\left(||\cdot||^R_{S^{\times}}\right),\Rami{pr_G}^*(||\cdot||_{G^{rss}}))=
    \\&=    \max(\Rami{pr_G^*}\left(p|_{\bfG^{rss}}^*(||\cdot||_{C^{rss}})\right), ||\cdot||^R_{Com^{rss}},\Rami{pr_G^*}(||\cdot||_{G^{rss}}))=
    \\&=\Rami{    \max(\Rami{pr_G^*}\left(\max(p|_{\bfG^{rss}}^*(||\cdot||_{C^{rss}}), ||\cdot||_{G^{rss}}) \right), ||\cdot||^R_{Com^{rss}})}\overset{\text{\Rami{Lem \ref{lem:pul.kot} \eqref{lem:pul.kot:1}}}}{\sim}
    \\&\sim
    \max(\Rami{pr_G^*}(||\cdot||_{G^{rss}}),||\cdot||^R_{Com^{rss}})=||\cdot||'_{Com^{rss}}
\end{align*}
\end{proof}
\subsection{Proof of \Cref{lem:Lev.NDP}}\label{ssec:pf.Lev.NDP}

Let $\lambda$ be a composition of $n$ such that $\bfM={}^\lambda\bfG$.

Consider the following diagram:
$$
\begin{tikzcd}
& {}^\lambda\bfG\arrow[r,"{}^\lambda p"] & {}^\lambda\bfC&
\\
& ({}^\lambda\bfG \cap \bfG^{rss})\arrow[r, "{}^\lambda p^{rss,cop}"]\arrow[u] \arrow[ur, phantom, "\square"]& {}^\lambda\bfC^{rss,cop}\arrow[u]&
\\
\bfG\times ({}^\lambda\bfG \cap \bfG^{rss}) \arrow[r,"\phi_\lambda"] \arrow[d,"i_1"] \arrow[dr, phantom, "\square"] & \bfG^{rss}  \times_{\bfC^{rss}} ({}^\lambda\bfG \cap \bfG^{rss})\arrow[d,"i_2"]\arrow[r," p'"]\arrow[u]\arrow[ur, phantom, "\square"]& \bfG^{rss}\times_{\bfC^{rss}} {}^\lambda\bfC^{rss,cop}\arrow[r,"\Rami{pr_1}"]\arrow[d,"\Rami{pr_2}"]\arrow[dr, phantom, "\square"]\arrow[u]&\bfG^{rss}  \arrow[d,"p^{rss}"]\\
\bfG\times \bfG^{rss} \arrow[r,"\phi"] & \bfG^{rss}\times \bfG& {}^\lambda\bfC^{rss,cop}\arrow[r,"m_\lambda^{C^{rss}}"]&\bfC^{rss}
\end{tikzcd}
$$
were:
\begin{itemize}
    \item $i_1,i_2$ are the natural embeddings.
    \item \Rami{$\phi$ is defined in \Cref{lem:phi}}
    \item the rest of the maps are either defined above \Rami{(see the index)} or defined by the Cartesian squares.
\end{itemize}
By \Cref{lem:mfin} the map 
$m_\lambda^{C^{rss}}$ has NDP. Therefore, by \Cref{lem:cart.NDP},  the map $\Rami{pr_1}$ has NDP. The companion matrix gives a section to ${}^\lambda p$. Therefore $p'$ also has a section. Thus $p'$ is onto on the level of $F$-points and, by \Cref{lem:kot.sec},  has NDP. Therefore, by \Cref{lem:Kot.comp}, the composition $ \Rami{pr_1} \circ p'$ has NDP. By \Cref{lem:phi} $\phi$ has NDP. Thus by \Cref{lem:cart.NDP}, the map $\phi_\lambda$ has NDP. Since any 2 regular semi-simple matrices with the same characteristic polynomial are conjugated, the map $\phi_\lambda$ is onto on the level of $F$-points. Therefore by \Cref{lem:Kot.comp}, the composition $ \Rami{pr_1} \circ p'\circ \phi_\lambda$ has NDP as required.


\section{Bound on averaging of a function - Proof of \Cref{thm:bnd.A}}\label{subsec:Pf.bnd.A}
The idea of the proof of \Cref{thm:bnd.A} is based on the analysis of the action map $\Rami{\phi_x:G^{ad}\to G}$. 

The main step is to bound the measure $(\Rami{\phi_x})_*(\mu_{G^{ad}}1_{(G^{ad})_i})|_\Rami{Ad(G) \cdot x}$ in terms of $\mu_{Ad(G) \cdot x}$. This is done in \Cref{cor:push} below.
\Rami{Let us start with some preparations}:
\Rami{
\begin{notation}
For $x\in G^{rss}$ let:
\begin{enumerate}[(i)]
    \item $G^{ad}_x:=G_x/Z(G)$.
    \item $\mu_{G^{ad}_x}$ be the Haar measure on $G^{ad}_x$ corresponding to the Haar measures $\mu_{G_x}$ and $\mu_{Z(G)}$.
    \item For $g\in G^{ad}$ let $\mu_{gG^{ad}_x}$ be the measure on the coset $gG^{ad}_x$ corresponding to the measure $\mu_{G^{ad}_x}$.
\end{enumerate}    
\end{notation}
}

\begin{lemma}\label{lem:meascomp}
     There exists a polynomial $\mdef{\alp_{vol}}\in\N[t]$  such that for any $x\in G^{rss}$ and $i\in \N$ we have
    $$\mu_{G^{\Rami{ad}}_x}(G^{\Rami{ad}}_x\cap (G^{ad})_i)<\alp_{vol}(i+ov_{G^{rss}}(x)).$$
\end{lemma}
\begin{proof}
The proof is based on the comparison of different norms on $Com^{rss}$, which is given by \Cref{prop:norm}.
\RamiA{We will use \Cref{not:norms.com} and the conventions given in \Cref{not:ov.and.ball}.}

Since $Com^{rss}$ is a closed subvariety of $G^{rss}\times G$, there exists $c\Rami{>1}$ such that for any $(x,y)\in Com^{rss},$ we have 
    $$c^{-1}\max(||x||_{G^{rss}},||y||_G)^{-c} < ||(x,y)||_{Com^{rss}}<c\max(||x||_{G^{rss}},||y||_G)^c.$$
    By \Cref{prop:norm} there exists $d\Rami{>1}$ such that
    $$d^{-1}||\cdot||_{Com^{rss}}^{-d}<||\cdot||'_{Com^{rss}}<d||\cdot||_{Com^{rss}}^d$$
Set $\alp_{vol}(i):=(2(n(d(ci+c)+d)+1))^n$. Let $x\in G^{rss}$ and $i\in \N$. We have
$$(G^\Rami{ad})_\Rami{i}\cap \Rami{G^{ad}_x}=((G_\Rami{i}\cap G_x)Z(G))/Z(G).$$
Define an embedding $\iota_x:G\hookrightarrow G\times G$ by $\iota_x(g):=(x,g)$. Then we have 
\begin{multline*}
    G_\Rami{i}\cap G_x\subset\iota_{\RamiA{x}}^{-1}\RamiA{(}(Com^{rss})_{c(ov(x)+i)+c}\RamiA{)}\sub \iota_{\RamiA{x}}^{-1}\RamiA{(}(Com^{rss})'_{d(c(ov(x)+i)+c)+d}\RamiA{)}\sub \\\iota_{\RamiA{x}}^{-1}\RamiA{(}(Com^{rss})^R_{d(c(ov(x)+i)+c)+d}\RamiA{)}=(G_x)'_{d(c(ov(x)+i)+c)+d}
\end{multline*}
Thus we obtain 
\begin{multline*}
\mu_{G^{\Rami{ad}}_x}(G^{\Rami{ad}}_x\cap (G^{\Rami{ad}})_i)\leq\mu_{G^{\Rami{ad}}_x}((G_x)'_{d(c(ov(x)+i)+c)+d}Z(G)/Z(G))\\
\leq \mu_{G_x}((G_x)'_{d(c(ov_{G^{rss}}(x)+i)+c)+d})\leq (2(n(d(c(ov_{G^{rss}}(x)+i)+c)+d)+1))^n<\alp_{vol}(ov_{G^{rss}}(x)+i)
\end{multline*}
\end{proof}
\begin{cor}\label{cor:al.vol}
With $\alpha_{vol}$ from \Cref{lem:meascomp}, the following holds:
for any
    \begin{itemize}
    \item 
$x\in G^{rss}$ 
\item $[g]\in G^{ad}$
\item $i\in \bN$ 
\end{itemize}
we have 
$$\mu_{\Rami{[g]G^{ad}_x}}\left(\Rami{\phi_x}^{-1}(\Rami{\phi_x}([g])) \cap \Rami{(G^{ad})_i}\right )\leq \alpha_{vol}(i+ov_{G^{rss}}(x)+ov_{\Rami{G^{ad}}}([g]))$$
\end{cor}
\begin{proof}
Fix $x,[g],i$ as above. WLOG we may assume that $||g||_G=||[g]||_{\Rami{G^{ad}}}$.
\Rami{
\begin{align*}
\left(\Rami{\phi_x}^{-1}(\Rami{\phi_x}([g])) \cap ({G^{ad}})_i \right)\cdot [g]^{-1} &= G^{ad}_x \cap (({G^{ad}})_i \cdot [g]^{-1}) \\&
\subset 
G^{ad}_x \cap ({G^{ad}})_{i+ ov_G(g) } 
\\&=
G^{ad}_x \cap ({G^{ad}})_{i+ ov_{{G^{ad}}}([g])} 
\end{align*}
}
So,  by \Cref{lem:meascomp}, 
\Rami{

\begin{align*}
\mu_{[g]G^{ad}_x}\left(\Rami{\phi_x}^{-1}(\Rami{\phi_x}([g])) \cap (G^{ad})_i\right)&=
\mu_{{G^{ad}_x}}\left(\left(\Rami{\phi_x}^{-1}(\Rami{\phi_x}([g])) \cap (G^{ad})_i\right) [g]^{-1}\right)\leq
\\ \leq &
\mu_{{G^{ad}_x}}\left(G^{ad}_x \cap ({G^{ad}})_{i+ ov_{{G^{ad}}}([g])}\right) 
\overset{\text{Lem \ref{lem:meascomp}}}{\leq}
\\ \leq &
\alpha_{vol}\bigl(i+ ov_{G^{ad}}([g])+ov_{G^{rss}}(x)\bigr)
\end{align*}
}
\end{proof}

\begin{cor}\label{cor:push}
There is a polynomial $\mdef{\alpha_{push}}\in\N[t]$ such that \RamiA{for all $i\in \N$ we have}
$$\sup\left(\frac{(\Rami{\phi_x})_*(\mu_{G^\Rami{ad}}1_{(G^\Rami{ad})_i})\Rami{|_{G\cdot x}}}{\mu_{G\cdot x}}\right ) <\alpha_{push}(i+ov_{G^{rss}}(x)).$$
Here, $1_{(G^\Rami{ad})_i}$ is the characteristic function of the set $(G^\Rami{ad})_i.$
\end{cor}
\begin{proof}
    Take $\alpha_{push}(i)=\alpha_{vol}(2i+1)$.
    Fix $i\in \N$ and $x\in G^{rss}$.
    Let $y\in \supp\left(\frac{(\Rami{\phi_x})_*(\mu_{\Rami{G^{ad}}}1_{(\Rami{G^{ad}})_i})\Rami{|_{G\cdot x}}}{\mu_{G\cdot x}}\right)$. We can find $[g]\in (\Rami{G^{ad}})_i$ such that $\Rami{\phi_x}([g])=y$. Now, we have 
    \begin{align*}        
    \frac{(\Rami{\phi_x})_*(\mu_{\Rami{G^{ad}}}1_{(\Rami{G^{ad}})_i})\Rami{|_{G\cdot x}}}{\mu_{G\cdot x}}(y)&=
    \frac{(\Rami{\phi_x})_*(\mu_{\Rami{G^{ad}}}1_{(\Rami{G^{ad}})_i})\Rami{|_{G\cdot x}}}{\mu_{G\cdot x}}(\Rami{\phi_x}([g]))=
    \\&=
    \mu_{\Rami{[g]G^{ad}_x}}((\Rami{\phi_x}^{-1}(\Rami{\phi_x}([g])) \cap (\Rami{G^{ad}})_i))\overset{\Rami{\text {Cor \ref{cor:al.vol}}}}{\leq}
    \\&\leq
\alpha_{vol}(i+ov_{G^{rss}}(x)+ov_{\Rami{G^{ad}}}([g]))\leq 
    \\&\leq
\alpha_{vol}(2i+ov_{G^{rss}}(x))<\alpha_{push}(i+ov_{G^{rss}}(x))
  \end{align*}        

\end{proof}

\begin{proof}[Proof of \Cref{thm:bnd.A}]
Take $\alpha_{av}:=\alpha_{push}$.
Using the last corollary (\Cref{cor:push}) we get
\begin{align*}
|\cA_i(m)(x)|&= |\langle \Rami{\phi_x}^*(m), \mu_{\Rami{G^{ad}}} 1_{(\Rami{G^{ad}})_i} \rangle|=
\\&=
|\langle m, (\Rami{\phi_x})_*(\mu_{\Rami{G^{ad}}} 1_{(\Rami{G^{ad}})_i}) \rangle|\leq
\\&\leq
\langle (|m|)\Rami{|_{G\cdot x}}, (\Rami{\phi_x})_*(\mu_{\Rami{G^{ad}}} 1_{(\Rami{G^{ad}})_i})\Rami{|_{G\cdot x}} \rangle\overset{\Rami{\text{Cor \ref{cor:push}}}}\leq
\\&{\leq}
\langle (|m|)|_{\Rami{G}\cdot x}, \alpha_{push}(i+ov_{G^{rss}}(x)) \mu_{\Rami{G}\cdot x}\rangle=
\\&=
\alpha_{push}(i+ov_{G^{rss}}(x)) \langle (|m|)|_{\Rami{G}\cdot x}, \mu_{\Rami{G}\cdot x}\rangle=
\\&=
\alpha_{av}(i+ov_{G^{rss}}(x)) \Omega(|m|)(x) 
\end{align*}
\end{proof}

\section{Application of  norm bounds}\label{ssec:norm.bnd}
In this section we explicate the results of \S\ref{sec:norm.bnd} in \RamiA{a} language suitable for further use.
\RamiA{We will use \Cref{not:norms.com} and the conventions given in \Cref{not:ov.and.ball}.}

We start with the following simple lemma:
\begin{lemma}\label{lem:cent.ball}
    Let $M<G$ be a standard Levi subgroup. Let $x\in M^{el}\cap G^{rss}$. Then $G_x=(G_x)'_{2n}Z(M)$.
\end{lemma}
\begin{proof}
    First note that $G_x=M_x$. As both RHS and LHS are products of factors that correspond to the blocks of $M$ we can prove the equality for each block separately.  So we can reduce the statement to the case when $M=G$ (note that the rank of each block might be smaller than $n$, but the statement becomes weaker when we enlarge $n$). 

    In this case $G_x\cong E^\times$ when $E$ is certain field extension of $F$ \RamiA{of degree $\leq n$,} and under this identification $Z(M)$ is identified with $F^\times$. The assertion follows now from the equality:
    $$E^\times=\{e\in E:\, \ell^{-2n}<|e|_{E}<\ell^{2n} \} F^\times$$
\end{proof}
The following Lemma gives a uniform bound on the size of the compact part of maximal tori in $G$. 
\begin{lemma}\label{lem:bnd.ell.cent}
There exists a polynomial $\mdef{\alp_{ell}}\Rami{\in\N[t]}$ such that for any standard Levi subgroup $M<G$, and  
any $x\in M^{ell}\cap G^{rss}$  we have 
$G_{\alp_{ell}(ov_{G^{rss}}(x))} Z(M)\supset G_x$.
\end{lemma}
\begin{proof}
By \Cref{prop:norm} there exists $d\in\R_{>1}$  
such that for any $x\in G^{rss}$  and $y\in G_x$ we have 
\begin{equation}
||y||_G < d(||y||'_{G_x}||x||_{G^{rss}})^d.
\end{equation}
Take $\alp_{ell}(j):=d(2n+1)+dj$.
We obtain:
$$G_{\alp_{ell}(ov_{G^{rss}}(x))} Z(M)\supset G_{d+2nd+d ov_{G^{rss}}(x)} Z(M) \supset (G_x)'_{2n} Z(M)=G_x$$
Where the last equality follows from \Cref{lem:cent.ball}.
\end{proof}

The following is a uniform version of \cite[Corollary of Theorem 18]{HC_VD} which is valid in arbitrary characteristic:
\begin{lemma}\label{lem:pull}
There exists a polynomial $\mdef{\alpha_{pull}}:\N\to \N$ such that for any
\begin{itemize}
    \item $i\in \N$,
    \item standard Levi $M \sub G$, and 
    \item $x\in M^{el}\cap G^{rss}$
\end{itemize}
we have 
$$\phi_x^{-1}(G_i) \sub G^{ad}_{\alpha_{pull}(i+ov_{G^{rss}}(x))}(Z(M)/Z(G))$$
\end{lemma}
\begin{proof}
    By \Cref{lem:phi} there exists $c\in\R_{>1}$ such that
    for any $x\in G^{rss}$  and $y\in Ad(G)\cdot x$ there is $g\in G$ such that:
\begin{eqnarray}    
        &gxg^{-1}=y\\        &||g||_G<c(||x||_{G^{rss}}||y||_{G^{rss}})^c
\end{eqnarray}

Take $$\alp_{pull}(j)=c+2cj+\alp_{ell}(j).$$
Fix $i,M,x$ as in the lemma. Let $z\in G$ such that $zxz^{-1}\in G_i$. We have to find $z_0\in Z(M)$ such that $z zsuch that_0\in G_{\alpha_{pull}(i+ov_{G^{rss}}(x))}$. By the above we can find $g$ such that
\begin{itemize}    
        \item $gxg^{-1}=zxz^{-1}$
        \item $||g||_G<c(||x||_{G^{rss}}\ell^i)^c$
\end{itemize}
Let $z_1=g^{-1}z$. We get that $z_1\in G_x$. By \Cref{lem:bnd.ell.cent} we can write $z_1=z_2 z_3$ such that $z_3\in Z(M)$ \Dima{and} $ov_G(z_2)\leq \alp_{ell}(ov_{G^{rss}}(x))$.
\Dima{Take $z_0=z_3^{-1}\in Z(M)$. Then}
\begin{align*}
||zz_0||_G&=||g z_1  z_3^{-1}||_G=||g z_2||_G\leq ||g||_G ||z_2||_G\leq \\
&\leq
c(||x||_{G^{rss}}\ell^i)^c   \ell^{\alp_{ell}(ov_{G^{rss}}(x))}  = \\
&=c(||x||_{G^{rss}})^{c} \ell^{ic+ \alp_{ell}(ov_{G^{rss}}(x))}\leq \ell^{\alp_{pull}(i+ov_{G^{rss}}(x))}
\end{align*}

\end{proof}

\section{$A$-cuspidal functions}\label{subsec:Acusp}
\begin{notation}
For any $a\in T$, denote 
$$\mdef{V_a}:=\left\{g\in G\,\,|\lim_{\RamiA{i\to \infty}} Ad(a)^{\RamiA{i}}(g)=1\right\}$$ and 
$$\mdef{P_a}:=\left\{g\in G\,\,|\text{ the sequence } \RamiA{\{Ad(a)^{\RamiA{i}}(g)\}_{i\in\N}} \text{ is bounded} \right\}.$$
It is well known (see e.g. \cite{Deligne76}) that $P_a=\bfP_a(F)$ for a corresponding parabolic subgroup $\bfP_a<\bfG$ and $V_a=\bfV_a(F)$ where $\bfV_a$ is the unipotent radical of $\bfP_a$.
We also consider 
$V_a$ as a subgroup of $G^{ad}$.
\end{notation}

\begin{definition}\label{def:A.cusp}
Let $A<T$ be a standard torus. \begin{enumerate}
    \item 
We say that a function $f\in C^{\infty}(G^{ad})$ 
is \tdef{$A$-cuspidal} if for any non-central $a\in A$, \Dima{and any }$x\in G$ we have $$\int_{V_a}f(xu) du=0,$$ where $du$ is a Haar measure on $V_a$.
\item Let $i\in \N$. We say that $f\in C^{\infty}(G^{ad})$ is \tdef[adapted]{$\RamiA{(A,i)}$-adapted} if it:
\begin{enumerate}
    \item is $A$-cuspidal,
    \item is right-$A/Z(G)$-invariant, and
    \item satisfies $\supp(f)\subset G^{ad}_i\cdot (A/Z(G))$.
\end{enumerate}
\end{enumerate}
\end{definition}

In this section we prove the following:
\begin{thm}\label{thm:cusp.stab}
There exists a polynomial $\mdef{\alpha_{stab}}:\N\to \N$ such that for any 
\begin{itemize}
\item standard torus $A<T$,
    \item integer $i\in \N$,
    \item $y\in G_i$, 
    \item   $(A,i)$-adapted function $f\in    C^{\infty}(G^{ad})$, and 
    \item $j>j_0:={\alpha_{stab}(i)}$
\end{itemize}
we have 
    $$\int_{G^{ad}_{j}} f(xy) dx=\int_{G^{ad}_{j_0}} f(xy) dx,$$
            \RamiA{where the integral is taken w.r.t. a Haar measure on $G^{ad}$.}

\end{thm}
Let us first describe the steps of the proof.
\begin{enumerate}
    \item It is easy to deduce the theorem from the fact that the integral of $f$ over $xK_iy$ vanishes whenever $x$ is far enough from the center (in comparison to $i$). This statement is \Cref{lem:HC20}.
    \item We deduce  \Cref{lem:HC20} from the case $y=1$. This case is \Cref{lem:HC20.y.1}.
    \item Since the support of $f$ is close to $A$, we can write $x=x'a$ where $a\in A$ and $x'$ is relatively small. 
    \item We use the right-$A$-invariance of $f$ in-order to replace the integral by an integral over $x'aKa^{-1}$.
    \item If $a$ is far away from any 
    proper standard subtorus $A'<A$, 
    then $aKa^{-1}$ is very similar to 
    $V_{a^{-1}}$ (see \S\ref{sssec:conj}). So we deduce the result 
    from an effective version of 
    cuspidality of $f$ (see  \Cref{cor:eff.cusp} below). 
    This case is treated in \Cref{lem:HC20.ex} below.    
    \item In order to deduce \Cref{lem:HC20.y.1} in the general case we use the last step and induction on the rank of $A$.
\end{enumerate}

We now perform these steps formally.
\subsection{Effective 
cuspidality}
\RamiA{In this subsection} we 
give effective version of cuspidality, see \Cref{cor:eff.cusp} below.
\begin{lemma}\label{lem:AV.ndp}
    Let $\bfA<\bfT$ be a standard torus and let  $a\in A$. Then the multiplication map $\bfV_a \times(\bfA/Z(\bfG)) \to \bfG^{ad}$ has the norm descent property.
\end{lemma}
\begin{proof}
Follows from  \Cref{lem:fin.ndp} as this map is a closed embedding.
\end{proof}

\begin{cor}\label{cor:int}
    There exists a polynomial $\mdef{\alp_{int}}\in\N[t]$ such that 
    for any 
    \begin{itemize}
        \item standard  torus $\bfA<\bfT$
        \item $a\in A$
        \item $i\in \N$
    \end{itemize}
    we have:    $$G^{ad}_i (A/Z(G)) \cap V_a\subset G^{ad}_{\alp_{int}(i)}.$$  
\end{cor}
\begin{proof}
\RamiA{Let $A<G$ be a standard torus.}    By \Cref{lem:AV.ndp}, for any $a\in A$, we can find $c_a\in\R_{>1}$ such that for any $b\in A/Z(G)$ and $u\in V_a$ we have:
    \begin{equation}        \label{eq:norm.u}
    ||u||_{G^{ad}}<c_a(||ub||_{G^{ad}})^{c_a}
    \end{equation}
    Let $a_1,\dots,a_N\in A$ be such that 
    $$\{V_{a_k}|k=1,\dots,N\}= 
    \{V_a|a\in A\}.$$ 
    Take $$\alpha^{\RamiA{A}}_{int}(i)=\sum_{k=1}^{N}(c_{a_k}+ ic_{a_k}).$$
    Fix $A,a,i$ as in the corollary. Let  $x\in (G^{ad}_i (A/Z(G))) \cap V_a$. We have to show that $ov_{G^{ad}}(x)\leq \alp_{int}(i)$.
    Write $x=yb$ where $y\in G^{ad}_i$ and $b\in A/Z(G)$. We have $y=x b^{-1}$. By \eqref{eq:norm.u} above $||x||_{G^{ad}}<c_a(||y||_{G^{ad}})^{c_a}\leq c_a i^{c_a}$. Therefore 
    $$ov_{G^{ad}}(x)\leq  c_a+ ic_a \leq \alp^{\RamiA{A}}_{int}(i).$$
\RamiA{Take $\alpha_{int}=\sum_A \alpha^A_{int}$, where $A$ ranges over all standard tori of $G$. We get that for any $\bfA,a,i$ as above we have $$(G^{ad}_i (A/Z(G))) \cap V_a\subset G^{ad}_{\alp^A_{int}(i)}\subset G^{ad}_{\alp_{int}(i)}.$$}
\end{proof}
\begin{notation}
    For $x\in G$ and a subgroup $H<G$ denote by ${\eta_{x,H}}:H\to G$  the map given by $$\mdef{\eta_{x,H}}(y):=xy.$$
\end{notation}
\begin{cor}\label{cor:eta.inv}
    There exists a polynomial 
     $\alp_{\eta}\in\N[t]$ such that for any 
    \begin{itemize}
        \item standard torus $A<T$,
        \item $a\in A$,
        \item $i\in \N$, and 
        \item $x\in G^{ad}_i$
    \end{itemize}
    we have:
    $$\eta_{x,V_a}^{-1}(G^{ad}_i(A/Z(G))) \subset G_{\alp_{\eta}(i)}.$$   
\end{cor}
\begin{proof}
Take $\alp_\eta(i)=\alpha_{int}(2i)$. By \Cref{cor:int} we have
    \begin{multline*}    
    \eta_{x,V_a}^{-1}(G^{ad}_i(A/Z(G)))= (x^{-1}G^{ad}_i (A/Z(G))) \cap V_a\subset \\(G^{ad}_{2i} (A/Z(G))) \cap V_a\subset G_{\alp_{int}(2i)}=G_{\alp_\eta(i)}.
    \end{multline*}
    \end{proof}

\begin{cor}\label{cor:eff.cusp}
    There exists a polynomial $\mdef{\alp_{cusp}}\in\N[t]$ such that 
    for any 
    \begin{itemize}
        \item standard torus $A<T$,
        \item $a\in A$,
        \item $i\in \N$,
        \item $\RamiA{(A,i)}$-adapted $f\in C^{\infty}(G^{ad})$,
        \item $x\in G^{ad}_i$, and
        \item a compact $\Omega\sub V_a$ such that $V_a\cap G^{ad}_{\alp_{cusp}(i)} \sub \Omega$
    \end{itemize}
    we have:    $$\int_{\Omega}f(xu)du=0,$$  
    \RamiA{where the integral is taken w.r.t. a Haar measure on $V_a$.}
    \end{cor}
\begin{proof}
    Take $\alp_{cusp}=\alp_{\eta}$. By the assumption and the last corollary \Rami{(\Cref{cor:eta.inv})} we have   
\begin{equation}\label{eq:included}
\begin{aligned}
    \supp(\eta_{x,V_a}^*(f)) &= \eta_{x,V_a}^{-1}(\supp(f))  \subset \eta_{x,V_a}^{-1}(G^{ad}_i (A/Z(G)))
    \\&\subset V_a \cap  G^{ad}_{\alp_{cusp}(i)} \subset \Omega.
\end{aligned}
\end{equation}
     Thus,
    $$\int_{\Omega}f(xu)du=-\left(\int_{V_a}f(xu)du-\int_{\Omega}f(xu)du\right)=-\int_{V_a\smallsetminus \Omega}f(xu)du=\int_{V_a\smallsetminus \Omega}\eta_{x,V_a}^*(f)(u)du\overset{\eqref{eq:included}}{=}0$$
\end{proof}
\subsection{Conjugation of congruence subgroups}\label{sssec:conj}
Here we study the behavior of $aK_ia^{-1}$ when $a$ is large.
\begin{lemma}\label{lem:iva}
For any $a\in T$  and $i\in \N_{>1}$ we have 
$$a^{-1}K_i a= (a^{-1}K_ia\cap K_0) (a^{-1}K_i a\cap V_a)$$
\end{lemma}
\begin{proof}
From Iwahori decomposition of $K_i$ we get $$K_i=(K_i\cap P_{a^{-1}})(K_i\cap V_a).$$ Thus
\RamiA{
\begin{multline*}   
a^{-1}K_i a= a^{-1}(K_i\cap P_{a^{-1}}) a a^{-1}(K_i\cap V_a)a\DimaA{\subset} (a^{-1}K_i a\cap K_i) (a^{-1}K_ia\cap \DimaA{V_a})\subset\\\subset (a^{-1}K_ia\cap K_0) (a^{-1}K_i a\cap V_a)
\end{multline*}
The opposite inclusion is obvious.
}
    
\end{proof}

\begin{lemma}\label{lem:conj}
There exists a polynomial 
$\mdef{\alp_{conj}}:\N\to \N$ such that for every $i\in \N$ and $y\in G_i$ we have: 
    $$K_{\alp_{conj}(i)} \subset y^{-1}K_{i}y$$
\end{lemma}
\begin{proof}
Take $\alp_{conj}(i)=(n+2)i$
using Cartan decomposition write $y=k_1 a k_2$ where $k_i\in K_0$ and $a\in T\cap G_i$. We have $$y^{-1}K_{i}y=k_2^{-1} a^{-1} k_1^{-1}K_{i}k_1 a k_2=k_2^{-1} a^{-1} K_{i} a k_2\supset k_2^{-1} K_{\alp_{conj}(i)}  k_2=K_{\alp_{conj}(i)}$$
\end{proof}

\begin{definition}
Let $A<T$ be a standard torus.
\begin{enumerate}
\item Recall that ${\vert\vert\cdot\vert\vert_{G/A}}:=pr_*(||\cdot||_G)$ where $pr:G\to G/A$ is the projection.
    \item  Denote by $\mdef{\fX(A)}$ the set of all proper standard subtori of $A$.
    \item For $x\in G$ define $\mdef{depth_A(x)}:=\RamiB{\min}_{A'\in \fX(A)} ov_{G/A'}(\RamiB{[x]})$.
\end{enumerate}
\end{definition}
The following lemma is straightforward.
\begin{lem}\label{lem:depth:mean}
    For any
    \begin{itemize}
        \item integer $i$,
        \item composition $\lambda$ of $n$,
        \item $a\in T_{\lambda}$ with $depth_A(a)>ni$, and  
        \item \RamiB{$j_1,j_2\in \{1,\dots,n\}$ lying in different parts of the composition  $\lambda$.}
    \end{itemize}
    we have $ov_{F^\times}(a_{j_1}/a_{j_2})>i$ 
    \RamiB{where $a_{i}$ is the $i$-th component of the diagonal matrix $a$.}
\end{lem}
\begin{cor}\label{lem:cont}
\DimaB{F}or any
\begin{itemize}    
    \item $i\in \N,$
    \item  standard torus $A<T$, and 
    \item $a\in A$ such that $depth_{A}(a)> ni$
\end{itemize}
we have 
$$a^{-1}K_i a\supset V_a \cap G_i.$$
\end{cor}

\subsection{Vanishing of the integral on $xK_iy$}
\RamiB{In this subsection} we prove vanishing of the integral on $xK_iy$, see  \Cref{lem:HC20} below.
\begin{lemma}\label{lem:HC20.ex}
\RamiB{T}here exists a polynomial $\mdef{\alpha_{deep}}\in\N[t]$ such that for any 
\begin{itemize}
\item \RamiB{standard tori $A'<A<T$,}
    \item $i\in\N_{>1}$,
    \item  $\RamiA{(A,i)}$-adapted function $f\in    C^{\infty}(G^{ad})$, and     
 \item       \RamiB{$x\in G^{ad}_i  (A'/Z(G))$  \DimaB{with} $depth_{A'}(x)> \alpha_{deep}(i)$}
\end{itemize}
we have 
    $$\int_{K_i^{ad}} f(xk) dk=0,$$
    \RamiA{where the integral is taken w.r.t. a Haar measure on $K^{ad}_i$}.
\end{lemma}
\begin{proof}
    Take $\alp_{deep}(i)=\RamiB{n}^\DimaB{2}\alp_{cusp}(i)+i$.
    \RamiB{Let $A,A',i,f,x$ be as in the Lemma.
    }    
    Write $x=ga$ with $g\in G^{ad}_i$ and $a\in \RamiB{A'}/Z(G)$. We have:
    $$depth_{\RamiB{A'}}(a)\geq \RamiB{n}^\DimaB{2}\alp_{cusp}(i)$$
\RamiB{    and hence 
    $$depth_{\RamiB{A'}}({a^{-1}})\geq \DimaB{n}\alp_{cusp}(i).$$
    }
    Thus, by \Cref{lem:cont}
    $$a K_i a^{-1}\cap V_{a^{-1}}\supset V_{a^{-1}}\cap G_{\alp_{cusp}(i)}.$$
    Therefore, \RamiA{by \Cref{cor:eff.cusp},} for any $h\in G_i^{\RamiA{ad}}$ we have 
    \begin{equation}\label{eq:int.on.conj.0}
      \int_{aK_i^{ad}a^{-1}\cap V_{a^{-1}}} f(hu) du=0.  
    \end{equation}
    \RamiA{Now,}
    \begin{align*}
       \int_{K_i^{ad}} f(xk) dk &=\int_{K_i^{ad}} f(gak) dk=\int_{K_i^{ad}} f(gaka^{-1}) dk=
       \\&=
       \int_{a K_i^{ad}a^{-1}} f(gk_1) dk_1\overset{\RamiA{\text{Lem \ref{lem:iva}}}}{=}\int_{(a K_i^{ad}a^{-1}\cap K_0)(\RamiA{a} K_i^{ad}a^{-1}\cap V_{a^{-1}})} f(gk_1) dk_1=
       \\&=
       \int_{a K_i^{ad}a^{-1}\cap K_0}
       \int_{aK_i^{ad}a^{-1}\cap V_{a^{-1}}} f(gk_2u) du dk_2\overset{\RamiA{\text{\eqref{eq:int.on.conj.0}}}}{=}0
    \end{align*}
    Here $dk,dk_1,dk_2,du$ are appropriate Haar measures.
\end{proof}

\RamiB{
\begin{lemma}\label{lem:HC20.y.1.ind}
Let $A'<A<T$ be standard tori. Then
there exists a  polynomial $\mdef{\alpha^{A',A}_{van0}}\in\N[t]$ such that for any 
\begin{itemize}
    \item $i\in\N_{>1}$,
    \item  an  $\RamiA{(A,i)}$-adapted function $f\in    C^{\infty}(G^{ad})$, \DimaB{and}         \item         $x\in G^{ad}_i (A'/Z(G))$ \DimaB{with}
        $$ov_{G^{ad}}(x)>{\alpha^{A',A}_{van0}(i)},$$ 
\end{itemize}
we have 
    $$\int_{K_i^{ad}} f(xk) dk=0,$$
    \RamiA{where the integral is taken w.r.t. a Haar measure on $K^{ad}_i$}.
\end{lemma}

\begin{proof}
    Define recursively:
        $$\alpha_{van0}^{A',A}= i+\sum_{A''\in \fX(A')} \alpha_{van0}^{A'',A}\circ \alpha_{deep}.$$
We will prove the lemma by induction on $A'$ w.r.t. the inclusion order. The base of the induction $A'=Z(G)$ is obvious. So we will now fix $A'$ and assume the statement for any $A''\in \fX(A')$.
Let $i,f,x$ be as in the Lemma.
We will show the required vanishing by analyzing 2 cases:
\begin{enumerate}[{Case} 1.]
    \item $depth_{A'}(x)>\alpha_{deep}(i)$.\\
    This case follows immediately from the previous lemma (\Cref{lem:HC20.ex})
    \item $depth_{A'}(x)\leq \alpha_{deep}(i)$.\\
    In this case one can find $A''\in \fX(A')$ such that $x\in G_{\alpha_{deep}(i)} (A''/Z(G))$. The assertion follows now from the induction hypothesis.
\end{enumerate}
\end{proof}
}
\begin{cor}[A version of {\cite[Theorem 10]{HC_VD}}]\label{lem:HC20.y.1}
There exists a  polynomial $\mdef{\alpha_{van0}}\in\N[t]$ such that for any 
\begin{itemize}
\item standard torus  $A<T$,
    \item $i\in\N_{>1}$,
    \item  an  $\RamiA{(A,i)}$-adapted function $f\in    C^{\infty}(G^{ad})$\DimaB{, and}  \item         $x\in G^{ad}_i$  \DimaB{with}
        $$ov_{G^{ad}}(x)>{\alpha_{van0}(i)},$$ 
\end{itemize}
we have 
    $$\int_{K_i^{ad}} f(xk) dk=0,$$
    \RamiA{where the integral is taken w.r.t. a Haar measure on $K^{ad}_i$}.
\end{cor}
\RamiB{
\begin{proof}
Take $$\alpha_{van0}=\max_A  \alpha_{van0}^{A,A}$$ where the maximum is over all standard tori.
Let $ A,i,f,x$ be as in the lemma.

    If $xK^{ad}_i\cap G^{ad}_i(A/Z(G))=\emptyset$, we are done. Otherwise we can assume without loss of generality that $x\in G_i^{\DimaB{ad}}(A/Z(G))$. In this case the assertion follows from the previous lemma (\Cref{lem:HC20.y.1.ind}).
\end{proof}
}

\begin{prop}    
[{A version of   \cite[Theorem 20]{HC_VD}}]\label{lem:HC20}
There exists a polynomial $\mdef{\alpha_{van}}\in\N[t]$ such that for any 
\begin{itemize}
    \item standard torus $A<T$, 
    \item $i\in\N_{>1}$,
    \item $y\in G^{ad}_i$, 
    \item  $\RamiA{(A,i)}$-adapted function $f\in    C^{\infty}(G^{ad})$, and
 \item         $x\in G^{ad}$  such that 
        $$ov_{G^{ad}}(x)>{\alpha_{van}(i)}$$ 
\end{itemize}
we have 
    $$\int_{K_i^{ad}} f(xky) dk=0,$$
    \RamiA{where the integral is taken w.r.t. a Haar measure on $K^{ad}_i$}.
\end{prop}

\begin{proof}
        We take $$\alp_{van}(i):=\alp_{van_0}(i+\alp_{conj}(\alp_{conj}(i))).$$
Fix:
\begin{itemize}
    \item $A,i,y,f$ as above.
 \item         $x\in G^{ad}$  such that
        $$ov_{G^{ad}}(x)>{\alpha_{van}(i)}$$ 
\end{itemize}
        Denote $i'=\alp_{conj}(i)$,  $i''=\alp_{conj}(i')$.
        By \Cref{lem:conj},
        we have $$K_{i''}^{ad}\sub y^{-1}K_{i'}^{ad}y \sub K_i^{ad}.$$
        Now, we have
\begin{align*}
\int_{K_i^{ad}} f(xky) dk&=\sum_{[z]\in K_i^{ad}/K_{i'}^{ad}}\int_{K_{i'}^{ad}} f(xz\RamiA{k_1}y)dk_1=
\\&= 
\sum_{[z]\in  K_i^{ad}/K_{i'}^{ad}}\int_{K_{i'}^{ad}} f(xzyy^{-1}\RamiA{k_1}y)dk_1=
\\&= 
\sum_{[z]\in  K_i^{ad}/K_{i'}^{ad}}\int_{y^{-1}K_{i'}^{ad}y} f(xz\RamiA{y}k_2)dk_2=
\\&= 
\sum_{[z]\in  K_i^{ad}/K_{i'}^{ad}}\left(\sum_{[w]\in (y^{-1}K_{i'}^{ad}y)/K_{i''}^{ad}} \int_{K_{i''}^{ad}} f(x\RamiA{zy}wk_3)dk_3\right)=0
\end{align*}
Here $dk,dk_1,dk_2,dk_3$ are appropriate Haar measures, and the last equality follows from \Cref{lem:HC20.y.1}.
\end{proof}


\subsection{Proof of \Cref{thm:cusp.stab}}
    Take $\alpha_{stab}:=\alpha_{van}$.
    Fix
\begin{itemize}
\item a standard torus $A<T$, 
    \item an integer $i\in \N$,
    \item $y\in G^{ad}_i$,
    \item  an  $(A,i)$-adapted function $f\in    C^{\infty}(G^{ad})$, and 
    \item $j>j_0:={\alpha_{stab}(i)}$
\end{itemize}
We have 
\begin{align*}   
    &\int_{G^{ad}_{j}} f(xy) dx-\int_{G^{ad}_{j_0}} f(xy) dx\\ &=   \int_{G^{ad}_{j}\smallsetminus G^{ad}_{j_0}} f(xy) dx=\sum_{[z]\in (G^{ad}_{j}\smallsetminus G^{ad}_{j_0})/K^{ad}_1}
    \int_{K_1^{ad}} f(zky) d\RamiB{k}=0
\end{align*}
Here the last equality follows from \Cref{lem:HC20}.

\section{Proof of \RamiB{Theorems} \ref{thm:stab} \RamiB{and \ref{thm:Om.bnd.Av}}}\label{subsec:PfStab}

\begin{lemma}\label{lem:com.pres.mes}
Let $M<G$ be a Levi subgroup. Let   
     $x\in M\cap G^{rss}$. Let 
     $A:=Z(M)$ and $a\in A$. Then:
     \begin{enumerate}
         \item\label{lem:com.pres.mes:1} For any $u\in V_a$ we have $uxu^{-1}x^{-1}\in V_a$.
         \item\label{lem:com.pres.mes:2} The map  $C_x:V_a\to V_a$ defined by $$C_x(u):=uxu^{-1}x^{-1}$$ is a homeomorphism that maps a Haar measure to a Haar measure. 
     \end{enumerate}
     
\end{lemma}
\begin{proof}
Item \eqref{lem:com.pres.mes:1} is obvious. Let us prove item \eqref{lem:com.pres.mes:2}.
Consider the natural filtration $V^k_a$ on the unipotent group  $V_a$. The quotient $V^k_a/V^{k+1}_a$ has a natural structure of a linear space. It is easy to see that $C_x$ preserves this filtration and acts as an invertible linear operator on each $V^k_a/V^{k+1}_a$. This implies the assertion.
\end{proof}

\begin{lemma}\label{lem:cusp.pull}
    Let $f\in C^{\infty}(G)$  be a cuspidal function. 
Let $M<G$ be a Levi subgroup. Let   
     $x\in M\cap G^{rss}$. Let 
     $A:=Z(M)$ and $a\in A$. 
Recall that $\phi_x:G^{ad}\to G$ is defined by $\phi_x([g])=gxg^{-1}$. 
    Then $\phi_x^*(f)$ is $A$-cuspidal (and $A$-right-invariant).   \end{lemma}
\begin{proof}
    Let $a\in A$ and $g\in G$. Let $[g]$ be the class of $g$ in $G^{ad}$.
    For $u\in V_a$ we have 
    \begin{align*}        
\phi_x([g]u)=&guxu^{-1}g^{-1}=guxu^{-1}x^{-1}\RamiB{g^{-1}g}xg^{-1}=gC_x(u)g^{-1}gxg^{-1}=\\=
&(Ad\RamiB{(g)}\circ C_x)(u)gxg^{-1}.
    \end{align*}    
    Therefore:
    $$\int_{V_a}\phi_x^*(f)([g]u)du=\int_{V_a} f((Ad\RamiB{(g)}\circ C_x)(u)gxg^{-1})du=
    \int_{Ad\RamiB{(g)}({V_a})}
    f(vgxg^{-1})dv,
    $$
    where $du$ is a Haar measure on $U$ and $dv=(Ad\RamiB{(g)}\circ C_x)_*(du).$ By the previous Lemma (\Cref{lem:com.pres.mes}) $dv$ is a Haar measure on $Ad\RamiB{(g)}(U)$. Since $f$ is cuspidal, we get $$\int_{V_a}\phi^*(f)([g]u)du=0$$ as required.
\end{proof}

\begin{proof}[Proof of \Cref{thm:stab}]
Let $N\in \N$ be such that $\supp(m)\subset \Dima{G_N}Z(G)$.
For any standard Levi $M<G$ let
$a_M:G\times (M\cap  G^{rss})\to G^{rss}$ be the action map. 
By \Cref{lem:Lev.NDP} we can find $c$ such that
for any standard Levi $M<G$  we have 
\begin{equation}\label{eq:thmC.1}
(a_M)_*(||\cdot||_G\times||\cdot||_{M\cap G^{rss}})<c((||\cdot||_{G^{rss}})|_{\Im(a_{\Rami{M}})})^c.
\end{equation}
Also, since $M\cap G^{rss}$ is a closed subset of $G^{rss}$, by \Cref{lem:fin.ndp}, we have  $d$ such that  for any standard Levi $M<G$ we have
\begin{equation}\label{eq:thmC.2}
||\cdot||_{G^{rss}}<d(||\cdot||_{M\cap G^{rss}})^d.
\end{equation}


We take 
$$\alpha^{\Dima{m}}_{ad-stab}(i)
:=
\alpha_{stab}(\alpha_{pull}(N+c\RamiB{d}+d+cdi)+c+ci).$$ 
Fix $x\in G^{rss}$ and $i>i_0:=\alpha^{\Dima{m}}_{ad-stab}(ov_{G^{rss}}(x))$.
We have to show that $\cA_i(m)(x)=\cA_{i_0}(m)(x)$.

Let $M<G$ be a standard Levi such that we have an element $x_0\in M^{el}\cap \RamiB{G}^{rss}$ which is conjugate to $x$.  Let $A=Z(M)$.
By \RamiB{\eqref{eq:thmC.1}}, we can find $x_1\in  M, y\in G$ such that
\begin{enumerate}[(a)]
    \item\label{it:thmC.a} $yx_1y^{-1}=x$
    \item\label{it:thmC.b} $||x_1||_{M\cap G^{rss}}<c(||x||_{G^{rss}})^{c}$
    \item\label{it:thmC.c} $||y||_G<c(||x||_{G^{rss}})^c$
\end{enumerate}
It is easy to see that $x_1\in M^{el}$.
\RamiB{By \eqref{eq:thmC.2},} we also have:
$$||x_1||_{G^{rss}}<d(||x_1||_{M\cap G^{rss}})^d$$
and thus:
\begin{equation}\label{eq:thmC.x1.bnd}
||x_1||_{G^{rss}}<c^\RamiB{d}d(||x||_{G^{rss}})^{cd}.    
\end{equation}


By \Cref{lem:pull}, 
we have
    $$\phi_{x_1}^{-1}(G_NZ(G)) =\phi_{x_1}^{-1}(G_N) \sub G^{\RamiB{ad}}_{\alpha_{pull}(N+ov_{G^{rss}}(x_1))} \RamiB{(}A/Z(G)\RamiB{)}.$$
Denote $i_1:=\alpha_{pull}(N+ov_{G^{rss}}(x_1))$.
By \Cref{lem:cusp.pull}, this implies that $\phi_{x_1}^{*}(m)$ is $(A,i_1)$-adapted.
For any $j\in \N$ 
we have:
$$A_j(m)(x)=\int_{G^{ad}_j}m(Ad(g)(x))dg\overset{\RamiB{\text{\eqref{it:thmC.a}}}}{=}
\int_{G^{ad}_j}m(Ad(g)(yx_1y^{-1}))dg=
\int_{G^{ad}_j}\phi_{x_1}^*(m)(g[y])dg,$$ where $[y]\in G^{ad}$ is the class of $y$.
Note that
\begin{align*}
i_0&=\alpha^{\Dima{m}}_{ad-stab}(ov_{G^{rss}}(x))
=
\alpha_{stab}(\alpha_{pull}(N+c\RamiB{d}+d+cd(ov_{G^{rss}}(x)))+c+c \cdot ov_{G^{rss}}(x))
\overset{\RamiB{\text{\eqref{eq:thmC.x1.bnd}}}}{\geq}
\\&\geq
\alpha_{stab}(\alpha_{pull}(N+ov_{G^{rss}}(x_1))+c+c \cdot ov_{G^{rss}}(x))
\overset{\RamiB{\text{\eqref{it:thmC.c}}}}{\geq}
\\&\geq
\alpha_{stab}(\alpha_{pull}(N+ov_{G^{rss}}(x_1))+ov_G(y))
\geq
\alpha_{stab}(i_1+ov_{G^{ad}}([y])).
\end{align*}
\RamiB{Let $i_2:=i_1+ov_{G^{rss}[y]}$. Note that $\phi_{x_1}^*(m)$ is $(A,i_2)$-adapted and $[y]\in G^{ad}_{i_2}$.}
So, by \Cref{thm:cusp.stab},
$$\cA_{i}(m)(x)=\int_{G^{ad}_i}\phi_{x_1}^*(m)(g[y])dg=\int_{G^{ad}_{i_0}}\phi_{x_1}^*(m)(g[y])dg
=\cA_{i_0}(m)(x)$$ as required.
\end{proof}

\begin{proof}[\RamiB{Proof of} \Cref{thm:Om.bnd.Av}]    
Take $\alpha^\RamiA{m}(i)=\alpha_{av}(\alpha_{ad-stab}^{m}(i)+i)$.
For any $x\in G^{rss}$ and $i\in \bN$ we have
\begin{align*}   
|\cA_i(m)(x)|&\overset{\text{Thm \ref{thm:stab}}}{\leq} \max_{k\leq \alpha_{ad-stab}^{m}(ov_{G^{rss}}(x))}|\cA_k(m)(x)| \leq 
\\&\overset{\text{Thm \ref{thm:bnd.A}}}{\leq} 
\max_{k\leq  \alpha_{ad-stab}^{m}(ov_{G^{rss}}(x)) }\alpha_{av}(k+ov_{G^{rss}}(x))\Omega(|m|)(x)=
\\&\,\,\,\,=\,\,\,\,
\alpha_{av}(\alpha_{ad-stab}^{m}(ov_{G^{rss}}(x)) +ov_{G^{rss}}(x))
\Omega(|m|)(x)=
\alpha^\RamiA{m}(ov_{G^{rss}}(x))\Omega(|m|)(x).
\end{align*}
\end{proof}

\Rami{
\section{Lie algebra versions of the main results}\label{sec:lie}
In this section we formulate  Lie algebra versions of the main results and explain how one can modify the proof of the main results in order to prove the Lie algebra versions.
\begin{definition}
    Let $f\in C^{\infty}(\fg)$.  
    \begin{itemize}
        \item 
   We say that $f$ is \tdef[cuspidal function]{cuspidal} if for any nilpotent radical $\fu$ of a proper parabolic subalgebra  of $\fg$ and
 any $x\in \fg$ the function $h:\fu\to \C$ given by $h(u):=f(x+u)$ is compactly supported and $$\int h\mu_\fu=0,$$ where $\mu_\fu$ is a Haar measure on $\fu$.    
        \item We denote the collection of cuspidal functions on $\fg$ by $\mdef{C^\infty(\fg)^{cusp}}$.
    \end{itemize}
\end{definition}
The proof of \Cref{thm:Om.bnd.Av} also gives:
\begin{customthm}{\ref{thm:Om.bnd.Av}'}
\label{thm:Om.bnd.Av.lie}
For any $m\in C^{\infty}(\fg)^{cusp}$ which has compact support modulo the center, there exists a polynomial $\mdef{\alpha^\RamiA{m}}\in\bN[t]$ such that 
for any $x\in G^{rss}$ 
we have
$$|\cA_i(m)(x)|\leq \alpha^\RamiA{m}(ov_{G^{rss}}(x))\Omega(|m|)(x).$$

Here we extend the definition of $\cA_i$ given in \Cref{def:Av} to functions on $\fg$ in the natural way.
\end{customthm}
To be more precise one has to modify 
the proof of \Cref{thm:Om.bnd.Av} as follows:
In \Cref{lem:com.pres.mes}, replace $C_x(u)$ with $D_x(u):=uxu^{-1}-x$. This allows to modify \Cref{lem:cusp.pull} to work for a cuspidal function $f\in C^\infty(\fg)$. The rest of the proof works with the obvious modifications.

\Cref{thm:main.lie} follows from \Cref{thm:Om.bnd.Av.lie} and the fact that $\hat\mu_x|_B=\lim \cA_i(m)|_B$ for some cuspidal $m\in C^{\infty}(\fg)$. This is proven exactly as in characteristic zero case, see \cite[Lemma 1.19]{HC_SD}.
}

\begingroup
  \let\clearpage\relax
  \let\cleardoublepage\relax 
  \printindex
\endgroup
\bibliographystyle{alpha}
\bibliography{Ramibib}

@article {AD,
    AUTHOR = {Aizenbud, Avraham and Drinfeld, Vladimir},
     TITLE = {The wave front set of the {F}ourier transform of algebraic
              measures},
   JOURNAL = {Israel J. Math.},
  FJOURNAL = {Israel Journal of Mathematics},
    VOLUME = {207},
      YEAR = {2015},
    NUMBER = {2},
     PAGES = {527--580},
      ISSN = {0021-2172},
   MRCLASS = {43A05},
  MRNUMBER = {3359711},
MRREVIEWER = {Radhakrishnan Nair},
       DOI = {10.1007/s11856-015-1181-9},
       URL = {https://doi.org/10.1007/s11856-015-1181-9},
}

@article {Deligne76,
    AUTHOR = {Deligne, Pierre},
     TITLE = {Le support du caract\`ere d'une repr\'esentation
              supercuspidale},
   JOURNAL = {C. R. Acad. Sci. Paris S\'er. A-B},
  FJOURNAL = {Comptes Rendus Hebdomadaires des S\'eances de l'Acad\'emie des
              Sciences. S\'eries A et B},
    VOLUME = {283},
      YEAR = {1976},
    NUMBER = {4},
     PAGES = {Aii, A155--A157},
      ISSN = {0151-0509},
   MRCLASS = {22E50},
  MRNUMBER = {425033},
MRREVIEWER = {G.\ I.\ Ol\cprime shanski\u i},
}

@book {HC_SD,
    AUTHOR = {Harish-Chandra},
     TITLE = {Admissible invariant distributions on reductive {$p$}-adic
              groups},
    SERIES = {University Lecture Series},
    VOLUME = {16},
      NOTE = {With a preface and notes by Stephen DeBacker and Paul J.
              Sally, Jr.},
 PUBLISHER = {American Mathematical Society, Providence, RI},
      YEAR = {1999},
     PAGES = {xiv+97},
      ISBN = {0-8218-2025-7},
   MRCLASS = {22E50 (17B15 22E35)},
  MRNUMBER = {1702257},
MRREVIEWER = {David\ Manderscheid},
       DOI = {10.1090/ulect/016},
       URL = {https://doi.org/10.1090/ulect/016},
}

@book {PR,
    AUTHOR = {Platonov, Vladimir and Rapinchuk, Andrei},
     TITLE = {Algebraic groups and number theory},
    SERIES = {Pure and Applied Mathematics},
    VOLUME = {139},
      NOTE = {Translated from the 1991 Russian original by Rachel Rowen},
 PUBLISHER = {Academic Press, Inc., Boston, MA},
      YEAR = {1994},
     PAGES = {xii+614},
      ISBN = {0-12-558180-7},
   MRCLASS = {11E57 (11-02 20Gxx)},
  MRNUMBER = {1278263},
}

@misc{SP,
  author       = {The {Stacks project authors}},
  title        = {The Stacks project},
  howpublished = {\url{https://stacks.math.columbia.edu}},
  year         = {2025},
}

@incollection {Kot,
    AUTHOR = {Kottwitz, Robert E.},
     TITLE = {Harmonic analysis on reductive {$p$}-adic groups and {L}ie
              algebras},
 BOOKTITLE = {Harmonic analysis, the trace formula, and {S}himura varieties},
    SERIES = {Clay Math. Proc.},
    VOLUME = {4},
     PAGES = {393--522},
 PUBLISHER = {Amer. Math. Soc., Providence, RI},
      YEAR = {2005},
      ISBN = {0-8218-3844-X},
   MRCLASS = {22E35 (17B99)},
  MRNUMBER = {2192014},
MRREVIEWER = {David\ A.\ Renard},
}

@book {HC_VD,
    AUTHOR = {Harish-Chandra},
     TITLE = {Harmonic analysis on reductive {$p$}-adic groups},
    SERIES = {Lecture Notes in Mathematics},
    VOLUME = {Vol. 162},
      NOTE = {Notes by G. van Dijk},
 PUBLISHER = {Springer-Verlag, Berlin-New York},
      YEAR = {1970},
     PAGES = {iv+125},
   MRCLASS = {22E50},
  MRNUMBER = {414797},
MRREVIEWER = {G.\ I.\ Ol\cprime shanski\u i},
}

@book{Ser92,
  author    = {Serre, J. P.},
  title     = {Lie Algebras and Lie Groups},
  series    = {Lecture Notes in Mathematics},
  volume    = {1500},
  year      = {1992},
  publisher = {Springer-Verlag},
  address   = {Berlin},
  note      = {1964 lectures given at Harvard University, reprinted with corrections}
}

@misc{AGKS3,
  title        = {On {H}arish-{C}handra's integrability theorem in positive characteristic},
  author       = {Aizenbud, Avraham and Gourevitch, Dmitry and Kazhdan, David and Sayag, Eitan},
  note         =  
 {Preprint available at 
  \href{https://www.wisdom.weizmann.ac.il/~dimagur/Publication_list.html}
  {{https://www.wisdom.weizmann.ac.il/$\sim$dimagur/}}
},}
\end{document}